\documentclass[preprint,3p]{elsarticle}

\usepackage{graphicx}
\usepackage{caption}
\usepackage{subcaption}
\usepackage{setspace}
\usepackage{amssymb}
\usepackage{amsmath, amsthm, amssymb}
\usepackage[mathscr]{euscript}
\usepackage{mathtools}
\usepackage{multirow}
\usepackage{float}

\theoremstyle{definition}
\newtheorem{definition}{Definition}
\newtheorem{theorem}{Theorem}
\newtheorem{lemma}{Lemma}

\newtheorem{corollary}{Corollary}
\newtheorem{remark}{Remark}

\begin{document}

\begin{frontmatter}

\title{Modified Lomax Model: A heavy-tailed distribution for fitting large-scale real-world
complex networks}

\author{Swarup Chattopadhyay $^{1,3}$}
\author{Tanujit Chakraborty \corref{cor1}$^2$}
\author{Kuntal Ghosh$^1$}
\author{Asit K. das$^3$}

\cortext[cor1]{Corresponding author: \\ Email:
tanujit\_r@isical.ac.in (Tanujit Chakraborty) \\ Tel. +91-33-2575
3104/3100, Fax
+91-33-2578-3357 \\
}
\address{$^{1}$ Machine Intelligence Unit, Indian Statistical Institute
Kolkata-700108, India \\
$^2$ SQC and OR, Indian Statistical Institute
Kolkata-700108, India. \\
$^3$ Department of C.S., Indian Institute of Engineering Science and
Technology, Shibpur, Howrah, India.\\
 }

\begin{abstract}

Real-world networks are generally claimed to be scale-free, meaning
that the degree distributions follow the classical power-law, at
least asymptotically. Yet, closer observation shows that the
classical power-law distribution is often inadequate to meet the
data characteristics due to the existence of a clearly identifiable
non-linearity in the entire degree distribution in the log-log
scale. The present paper proposes a new variant of the popular
heavy-tailed Lomax distribution which we named as the Modified Lomax
(MLM) distribution that can efficiently capture the crucial aspect
of heavy-tailed behavior of the entire degree distribution of
real-world complex networks. The proposed MLM model, derived from a
hierarchical family of Lomax distributions, can efficiently fit the
entire degree distribution of real-world networks without removing
lower degree nodes as opposed to the classical power-law based
fitting. The MLM distribution belongs to the maximum domain of
attraction of the Frechet distribution and is right tail equivalent
to Pareto distribution. Various statistical properties including
characteristics of the maximum likelihood estimates and asymptotic
distributions have also been derived for the proposed MLM model.
Finally, the effectiveness of the proposed MLM model is demonstrated
through rigorous experiments over fifty real-world complex networks
from diverse applied domains.
\end{abstract}

\begin{keyword}
Complex networks; Degree distribution; Lomax distribution;
Heavy-tailed distribution; Power-law; Statistical properties.
\end{keyword}

\end{frontmatter}


\section{Introduction}

The modeling and structural aspects of large scale real-world
complex networks, including social, information, collaboration,
communication, etc. have been well studied during the past decade
\cite{albert2002statistical,albert2000error,newman2001structure,newman2003structure,chacoma2018dynamical}
by many researchers. The World Wide Web, Twitter, Orkut, Youtube,
DBLP, Wiki talk, Facebook, LinkedIn are examples of such large scale
real-world complex networks. These networks are characterized by
several important structural, emergent properties like degree
distribution, correlation coefficient, average nearest neighbor,
average path length, clustering coefficient, community structure,
etc. Recently, the modeling and statistical aspects of such emergent
structural properties, therefore, remain an important research area
in the study of large scale real-world complex networks
\cite{zarandi2018community,newman2003structure,cui2014detecting,nie2016dynamic,shakibian2018statistical}.
In this regard, the node degree distribution has been well studied
and viewed as an important structural characteristic of real-world
networks \cite{muchnik2013origins}. In 1999, Barabasi and Albert
\cite{barabasi1999emergence,albert1999diameter} modeled the node
degree distribution of the World Wide Web (WWW) using a power-law.
Since then, many researchers have also favored the use of heavy
tailed power-law in modeling the node degree distribution of
real-world networks such as collaboration networks, communication
networks, social networks, biological networks, etc
\cite{clauset2009power,barabasi2005origin}.  Mathematically, a
quantity $x$ follows a power-law if it is drawn from a probability
distribution $P(x) \propto x^{-\alpha}$ where, the parameter
$\alpha$ is a positive constant and is known as exponent or scaling
parameter of the distribution. Thus it is common to encounter the
claim that most of the real-world networks are scale-free, meaning
that the degree distributions follow single power-law.  Despite
this, a closer observation, while fitting, shows that the classical
power-law distribution is  often  inadequate  to  meet  the  data
characteristics adequately because of the existence of an
identifiable non-linearity (bend) when the entire degree
distribution is considered in log-log  scale as shown in Figures
\ref{fig:sub-first} and \ref{fig:sub-second} (elaborated later).

This feature (non-linearity) of the entire degree distribution,
depending on when and where it is considered or ignored, possibly
constitute the reason why the universality vis-a-vis scarcity of
scale-free networks has remained controversial ever since its
inception \cite{liljeros2001web,jones2003sexual}. The debate has
continued to crop up time and again throughout the last twenty-one
years \cite{newman2005power,
seshadri2008mobile,clauset2009power,barabasi2005origin} and in very
recent times too whence it has been claimed through an empirical and
extensive study that the power-law distribution does not fit well in
most cases and thereby produces a significant fitting error,
followed by counter-claims \cite{broido2019scale}.

This apart, researchers have also argued differently in favor of
scale-free structure while suggesting some softer statistical
criteria for scale-freeness \cite{holme2019rare, voitalov2019scale,
stumpf2012critical}. Especially significant in this context is the
following quote \cite{stumpf2012critical}: "The fact that
heavy-tailed distributions occur in complex systems is certainly
important (because it implies that extreme events occur more
frequently than would otherwise be the case)... However, a
statistically sound power-law is no evidence of universality without
a concrete underlying theory to support it. Moreover, knowledge of
whether or not a distribution is heavy-tailed is far more important
than whether it can be fit using a power-law".

Several other heavy-tailed distributions such as lognormal, Pareto
lognormal (PLN), double Pareto lognormal (DPLN), etc. also have been
proposed in modeling the degree distribution of real-world networks
instead of power-law \cite{seshadri2008mobile,sala2010brief}. Recent
research also recognized the deviations from a pure power-law
distribution over various network data sets and recommended some
other distributions for better modeling the heavy-tailed node degree
distribution
\cite{voitalov2019scale,chattopadhyay2019finding,chattopadhyay2014fitting}.
Thus, identifying the reasons for deviation of single power-law
while fitting and looking for the alternative models which  can
efficiently  capture  the crucial aspect of heavy-tailed  and
long-tailed  behaviour  of  the  entire degree distribution of
real-worlld complex networks continue to remain a challenging task
of current research in the field of complexity science even as it
steadily gravitates toward data science \cite{holme2019rare,
voitalov2019scale, stumpf2012critical}.

 \begin{figure}[ht]
 \begin{subfigure}{0.5\textwidth}
   \centering
   \includegraphics[width=0.9\linewidth]{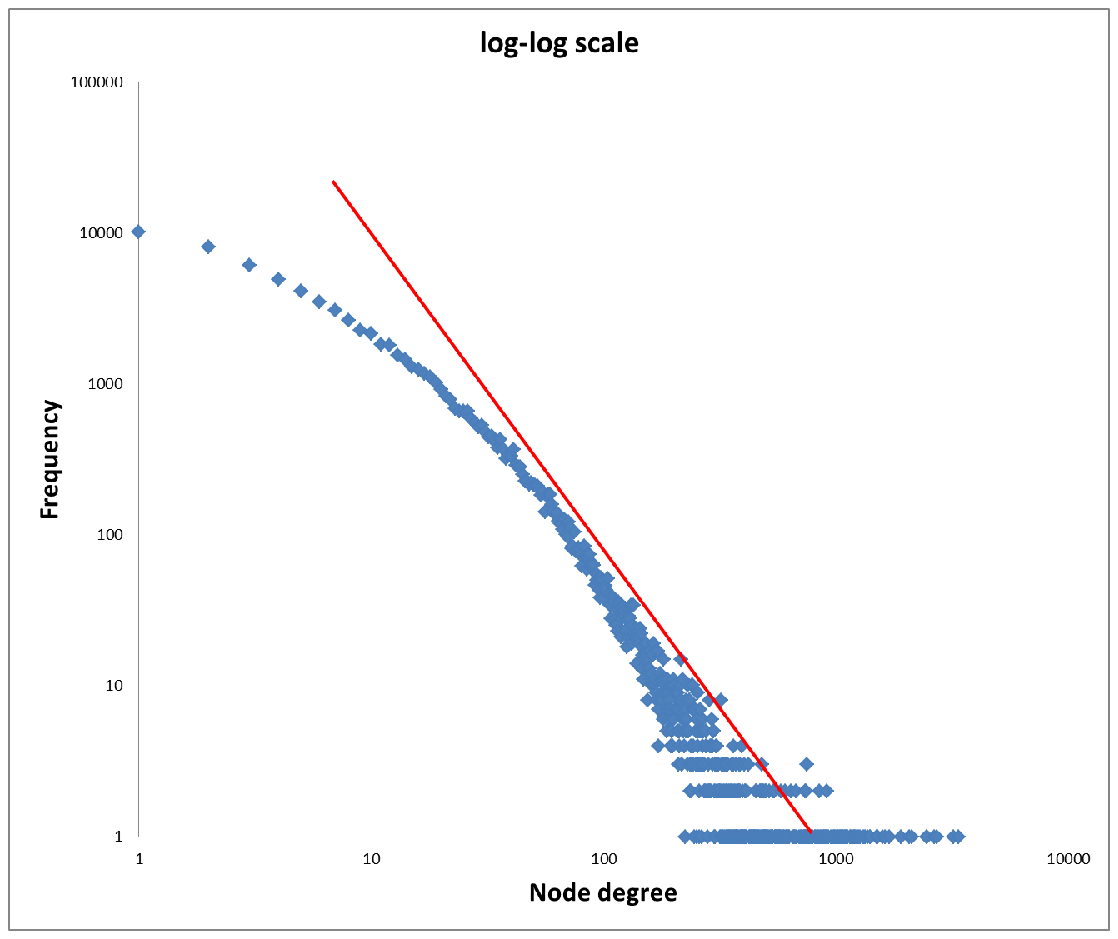}
   \caption{Twitter network}
   \label{fig:sub-first}
 \end{subfigure}
 \begin{subfigure}{0.5\textwidth}
   \centering
   \includegraphics[width=0.9\linewidth]{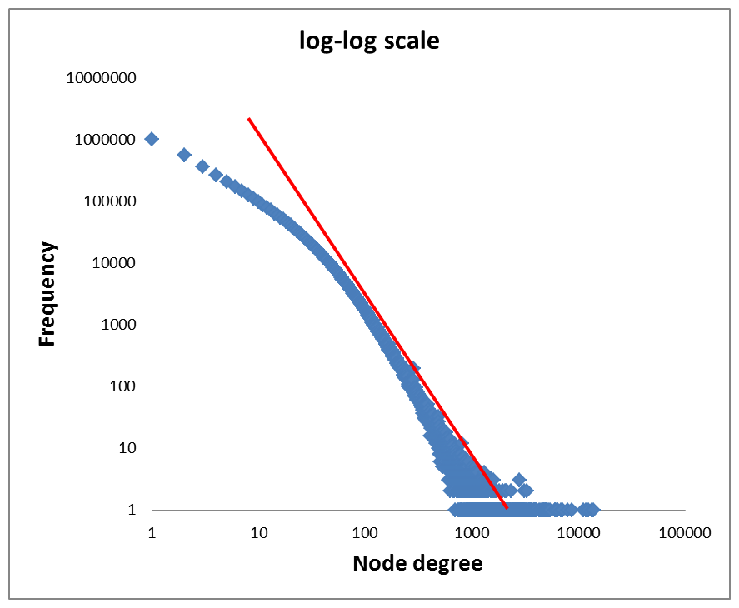}
   \caption{LiveJournal network}
   \label{fig:sub-second}
 \end{subfigure}
 \caption{Plot of degree distribution}
 \label{fig:fig}
 \end{figure}

\textit{Motivation:} Networks are a powerful way to represent and
study the structure of real-world complex systems. Across various
applied domains of networks, it is common to encounter the claim
that most of the real-world networks are scale-free, meaning that
the degree distributions follow single power-law, though the
universality of scale-free networks remains controversial as already discussed above.

Now consider an example where Figures \ref{fig:sub-first} and
\ref{fig:sub-second} that depict the plot of entire degree
distribution in the log-log scale of the Twitter and LiveJournal
social networks. The horizontal axis represents the unique degree
value ($x$), and the vertical axis represents the corresponding
frequency. In these networks, a node represents a single user, and
an edge represents a follower of that user. From these figures, it
is clear that the pattern of the degree distribution of these
networks does not match with the straight-line representation in the
log-log scale through a single power-law. Usually, while fitting the
node degree distribution, the single power-law is applied only for
values of degree higher than some minimum (say, $x_{min}$) and the
exponent $\alpha$ is estimated from the data using MLE accordingly.
Thus power-law distribution provides better fitting or in other
words better inclined to the right tail of the data unless
otherwise, some ``unimportant" (i.e., lower degree) nodes are left
out. Analytically, we can say that this inadequacy of fitting a
single power-law occurs because of nonlinear behavior of the degree
distribution curve in the log-log scale. This motivates the
researchers to use other heavy-tailed probability models with
non-negative exponent for better modeling the entire degree
distribution of real-world networks. To capture these nonlinearities
in the degree distribution of the real-world complex networks in a
log-log scale, previous studies used various heavy-tailed
probability distributions \cite{voitalov2019scale,
seshadri2008mobile, chattopadhyay2019finding}. In this current
research, we study the behavior of the entire degree distributions
with a new variant of Lomax distribution that has wide applications
in the field of actuarial science, reliability modeling, economics
and computer science
\cite{lomax1954business,ahsanullah1991record,hassan2016optimum,childs2001order}.
The Lomax distribution is essentially a Pareto Type-II distribution
that has been shifted so that its support begins at zero
\cite{ahsanullah1991record,lomax1954business}. Some extension and
generalization of the Lomax distribution has been carried out for
analyzing reliability and survival data sets in the past
\cite{ahsanullah1991record,al2001statistical,balakrishnan1994relations}.
Recent research also focused on a new generalization of Pareto
distribution with application to the breaking stress data
\cite{jayakumar2020new}. This paper proposes a modified Lomax (MLM)
distribution to be derived from a hierarchical family of Lomax
distributions where the non-negative shape parameter is assumed to
be expressible as a nonlinear function of the data.

\textit{Our contribution:} The major contribution here is to develop
a modified Lomax (MLM) distribution from a hierarchical
family of Lomax distributions for efficient modeling of the
entire degree distribution of real-world complex networks
\cite{ahsanullah1991record,lomax1954business}. The reasons for
introducing MLM distribution is to provide greater flexibility and
better fitting to the entire node degree distribution of complex
networks compared to other popularly used heavy-tailed
distributions. In other words, the proposed MLM model can be
used for effective modeling the degree distribution of complex
networks, coming from different disciplines, in the whole range of
the data without discarding some of the lower degree nodes.
Moreover, some statistical properties including extreme value and
asymptotic behavior of the proposed MLM distribution have been
studied in this context. We also provide mathematical arguments to
explain the behavior of the likelihood surface for this nonlinear
variant of the Lomax distribution, i.e., MLM distribution.
A sufficient condition for the existence of the global maximum for
the likelihood estimates is given using the notion of the
coefficient of variations (CV) and discuss the
parameter estimation procedures of the proposed MLM distribution. In order to justify the effectiveness of the proposed MLM distribution, we have compared it with the other common power-law-type distributions, viz. power-law, Pareto, lognormal, exponential, power-law with exponential cutoff and
Poisson \cite{sala2010brief,clauset2009power,newman2005power}. The
goodness-of-fit of the observed degree distribution is evaluated and
compared using a few statistical measures, viz. bootstrap
Chi-square, KL-divergence (KLD), mean absolute error (MAE) and root
mean square error (RMSE). Several real-world complex networks from
diverse fields have been used for experimental
evaluation. Empirical results confirm the effectiveness of the
proposed MLM distribution compared to other common distributions.

The remainder of the paper is organized as follows. Section \ref{model}
provides the details of the hierarchical family of Lomax distributions. We propose and interpret proposed modified Lomax (MLM) distribution in Section \ref{proposed_model}. Section \ref{statistics_prop} discusses
the statistical properties, including extreme value and asymptotic
behaviors of the proposed MLM distribution. Section \ref{experiments} is devoted to the experimental results with a detailed analysis of the results over several real-world complex networks. Finally, Section \ref{conclusion} concludes the paper with a brief discussion.

\section{Model}\label{model}
In this section, we first introduce a new family of heavy-tailed
Lomax (HLM) distributions. Further, we propose a relevant model from
this newly introduced family to model the real-world heavy-tailed
network data sets in the whole range.

\subsection{Genesis}
Lomax distribution has been used as an alternative to exponential,
power-law, gamma and weibull distribution for modeling heavy tailed
data sets \cite{atkinson1978distribution, bryson1974heavy,
chahkandi2009some, hassan2009optimum}. The cumulative distribution
function (CDF) and the probability density function (PDF) of the
Lomax model are defined as follows:

\begin{definition}
A random variable $Z$ follows Lomax distribution with parameters
$\alpha$ and $\sigma$ if the CDF is of the form:
$$F(z) = 1 - \bigg( 1+ \frac{z}{\sigma} \bigg)^{-\alpha}; \;
z \geq 0,$$ where $\alpha \; (>0)$ is the shape parameter (real) and
$\sigma \; (>0)$ is the scale parameter (real). The corresponding
PDF is defined as follows:
\begin{equation}
 f(z) = \frac{\alpha}{\sigma}\bigg( 1+
 \frac{z}{\sigma} \bigg)^{-\alpha-1}; \; z \geq 0 \label{1}
 \end{equation}
\end{definition}

 Below we introduce a new family of heavy tailed Lomax distributions which is right tail-equivalent to a power-law distribution.

\begin{definition}
A continuous random variable $X$ follows a family of heavy-tailed
Lomax (HLM) distributions if and only if it has the following CDF:
\begin{equation}
 F(x) = 1- \big( 1+ x \big)^{-m(x)}; \; x \geq 0 \label{hlm}
 \end{equation}
and $F(x)=0$ if $x\leq 0$, where $m:(0,\infty) \rightarrow
\mathbb{R}^+$ is a real, continuous, positive function which is
differentiable on $(0,\infty)$ and satisfies the following
conditions:
\begin{enumerate}
\item The function $m$ is strictly positive and have finite limit at infinity, i.e., $\displaystyle \lim_{x \to \infty} m(x) = \alpha \; (> 0). $
\item $\displaystyle \lim_{x \to 0^+} (1+x)^{m(x)} = 1 \quad \text{and} \quad  \lim_{x \to \infty } (1+x)^{m(x)} = \infty. $
\item $\displaystyle \frac{m^{'}(x)}{m(x)} \geq - \frac{1}{(1+x)\log(1+x)} , \; x > 0.$
\end{enumerate}
\end{definition}

It can be easily verified that the CDF in (\ref{hlm}) satisfying
conditions (1), (2) and (3) is a genuine CDF which can also be
expressed as follows:
\[
F(x)= 1 -\exp{\left [-m(x)\log (1+x)\right ], \; x > 0}
\]
The PDF of this new family of heavy-tailed Lomax distribution is of
the form:
\[
\displaystyle f(x)= (1+x)^{-m(x)} \left [ \frac{m(x)}{(1+x)} + m'(x)
\log (1+x) \right], \quad x > 0 \quad \text{and} \quad f(x)=0, \quad
x \leq 0.
\]
There can be a wide variety of choices of $m(x)$ satisfying
$\displaystyle \lim_{x \to \infty} m(x) = \alpha \; ( > 0 ) $. It is
noted that the simplest choice of $m(x)=\alpha$ and
$x=\frac{z}{\sigma}$ corresponds to the Lomax distribution. We
further represent this newly introduced family of Lomax
distributions as a hierarchical family in accordance with Pareto
distribution \cite{arnold2015pareto}.

\begin{definition}{(HLM Type-I family of distributions)}
Supposed that a random variable $X$  folows HLM family of
distributions as defined in (\ref{hlm}). Then with a scale parameter
$\sigma \in (0, \infty)$, the CDF of HLM Type-I family of
distributions takes the following form:
\[
F(x)=1-\left[1+ \left( \frac{x}{\sigma}-1 \right)\right]^{-m\left( \frac{x}{\sigma} -1
\right )} , \quad x > \sigma
\]
By taking $m\left( \frac{x}{\sigma} -1 \right ) = \alpha \; (> 0)$,
we obtain the classical Pareto Type-I distribution.
\end{definition}

\begin{definition}{(HLM Type-II family of distribution)}
Supposed that a random variable $X$ folows HLM family of
distributions as defined in (\ref{hlm}). Then with a location
parameter $\mu \in \mathbb{R}$ and a scale parameter $\sigma \in (0,
\infty)$, the CDF of HLM Type-II family of distributions takes the
following form:
\[
F(x)=1-\left ( 1+ \frac{x-\mu}{\sigma} \right )^{-m\left(
\frac{x-\mu}{\sigma}\right )} , \quad x > \mu
\]
By taking $m\left( \frac{x-\mu}{\sigma}  \right ) = \alpha \; (>
0)$, we obtain the Pareto Type-II distribution. Also, in addition
$\mu = 0$ corresponds to the Lomax distribution.
\end{definition}

\begin{definition}{(HLM Type-III family of distribution)}
Supposed that a random variable $X$ folows HLM family of
distributions as defined in (\ref{hlm}). Then with a location
parameter $\mu \in \mathbb{R}$, scale parameters $\sigma \in (0,
\infty)$ and a shape parameter $\gamma \; (> 0)$, the CDF of HLM
Type-III family of distributions takes the following form:
\[
F(x)=1- \left[ 1+ \left(\frac{x-\mu}{\sigma} \right
)^{\frac{1}{\gamma}} \right ]^{-m \left [ \left(
\frac{x-\mu}{\sigma}\right )^{\frac{1}{\gamma}} \right ] } , \quad x
> \mu
\]
By taking $\left[ m \left( \frac{x-\mu}{\sigma}\right
)^{\frac{1}{\gamma} } \right]= 1$, we obtain the Pareto Type-III
distribution.
\end{definition}
Obviously, the choice of $m(\cdot)$ function is subjective and any
function $m$ satisfying conditions (1), (2) and (3) will give some
known (unknown) heavy-tail Lomax distributions.

\section{Modified Lomax (MLM) Model} \label{proposed_model}

The Lomax distribution does not provide great flexibility in
modeling heavy-tailed data sets in the whole range similar to the
power-law distribution. Due to this, the trend of parameter(s)
induction to the baseline Lomax distribution has received increased
attention in the recent years. Several generalized classes of
distributions by adding additional parameters such as shape and or
scale and or location in the distribution are available such as
exponentiated Lomax (EL) \cite{abdul2012recurrence}, Beta-Lomax (BL)
\cite{rajab2013five}, exponential Lomax (ELomax)
\cite{el2015exponential}, Gamma-Lomax (GL) \cite{cordeiro2015gamma}
and Gumbel-Lomax (GuLx) model \cite{tahir2016gumbel}.

This paper provides a new modified version of the Lomax distribution
called modified Lomax (MLM) distribution. MLM distribution is shown
to be an asymmetric distribution, which provides great fit in
modeling large-scale heavy-tailed data sets. The proposed MLM model
is derived from the HLM family of distributions (in particular, HLM
Type-II model) that can efficiently model the entire degree
distribution of real-world networks. In other words, the proposed
MLM model can be used for effective modeling the degree distribution
of real-worlld complex networks in the whole range without
discarding lower degree nodes. We define a relevant model from the
newly introduced HLM Type-II family with the location parameter $\mu
= 0$ and we choose a flexible $m(\cdot)$ function that depends on
two shape parameters $\alpha$ and $\beta$ satisfying  $
\displaystyle \lim_{x \to \infty} m(x) = \alpha$. The rational
behind adding an additional shape parameter in the HLM Type-II
family of distribution will make the statistical model more
flexible, simple and have physical interpretation. This idea of
generalization should suffice the practical needs of working with
the non linear exponent to address the structural issue (degree
distribution) of real-world complex networks.

Now we choose a nonlinear function $m$ that adds a nonlinear exponents while fitting heavy-tailed HLM Type-II model in the degree distributions is as
follows:
$$ m(x) = \alpha \Bigg( \frac{\log \big( 1 + {x}
\big)}{1+\log \big( 1 + x \big)} \Bigg)^{\beta}.$$
The chosen $m(x)$ approaches to $\alpha$ from below if
$-1<\beta<0$ as $x \to \infty$ and approaches to $\alpha$ from above
for $\beta > 0$ as $x \to \infty$. Note that, the function $m(x)$ as
defined above includes the constant function (in this case $\alpha$)
as special cases by setting $\beta = 0$. The derivative of $m(x)$ is
given by
$$ m'(x) = \frac{\alpha\beta}{x+1} \Bigg( \frac{\log
\big( 1 + {x} \big)}{1+\log \big( 1 + {x} \big)} \Bigg)^{\beta
-1}.\bigg(1+\log \big( 1 + {x} \big)\bigg)^{-2}. $$
Now, we define a relevant model with the above choice of $m(\cdot)$
in the HLM Type-II model with $\mu=0$ and name it as Modified Lomax
Model to be denoted by $MLM(\alpha,\beta,\sigma)$. This modification
to the Lomax distribution provides more flexibility in the data
modeling since the non-negative shape parameter are assumed to be
expressed as a nonlinear function of the empirical data. Thus the
proposed MLM model with parameters $\alpha,\beta,\sigma$ could be
useful for modeling the heavy-tailed degree distribution of
real-world complex network data sets in the whole range.

\begin{definition}{\textbf{(Modified Lomax Distribution)}}
A continuous random variable $X$ follows $MLM(\alpha,\beta,\sigma)$
distribution with $\alpha \; (> 0)$ and $\beta \; (> -1)$ as the
shape parameters and $\sigma \; (>0)$ as the scale parameter if the
CDF takes the following form:
\begin{equation}\label{gplcdf}
F(x)= 1 -
\exp{\left[-\alpha\displaystyle\frac{\log^{\beta+1}(1+x/\sigma)}{[1+\log(1+x/\sigma)]^\beta}\right]},
\; x>0,
\end{equation}
and $F(x)=0$ if $x\leq 0$. The corresponding PDF is given by,
\begin{equation}\label{gplpdf}\small
f(x)=
\displaystyle\frac{\alpha\left[\beta+1+\log\left(1+\frac{x}{\sigma}\right)\right]
\left[\log\left(1+\frac{x}{\sigma}\right)\right]^\beta}
{\sigma\left(1+\frac{x}{\sigma}\right)
\left[1+\log\left(1+\frac{x}{\sigma}\right)\right]^{\beta+1}}
\exp\left[-\alpha\displaystyle\frac{\left[\log\left(1+\frac{x}{\sigma}\right)\right]^{\beta+1}}
{\left[1+\log\left(1+\frac{x}{\sigma}\right)\right]^\beta}\right],\;x>0
\end{equation}
and $f(x)=0$ if $x\leq 0$.
\end{definition}

This MLM model includes Lomax distribution $\beta=0$ as particular
case. In addition, it belongs to the new family of HLM Type-II
distribution satisfying the condition: $\displaystyle\lim_{x \to
\infty}m(x)=\alpha \; (>0)$. Due to the addition of an additional
parameter $\beta$ in the exponents of the Lomax distribution
generates various shapes (unimodal and bimodal) and provides greater flexibility (nonlinearity and heavy-tail) as shown in Figure \ref{fig:pdf_plot}. We study the
monotonicity for the PDF of the proposed MLM model in Theorem
\ref{mono} below.

\begin{figure}[h]
\begin{subfigure}{.5\textwidth}
  \centering
  \includegraphics[width=.9\linewidth]{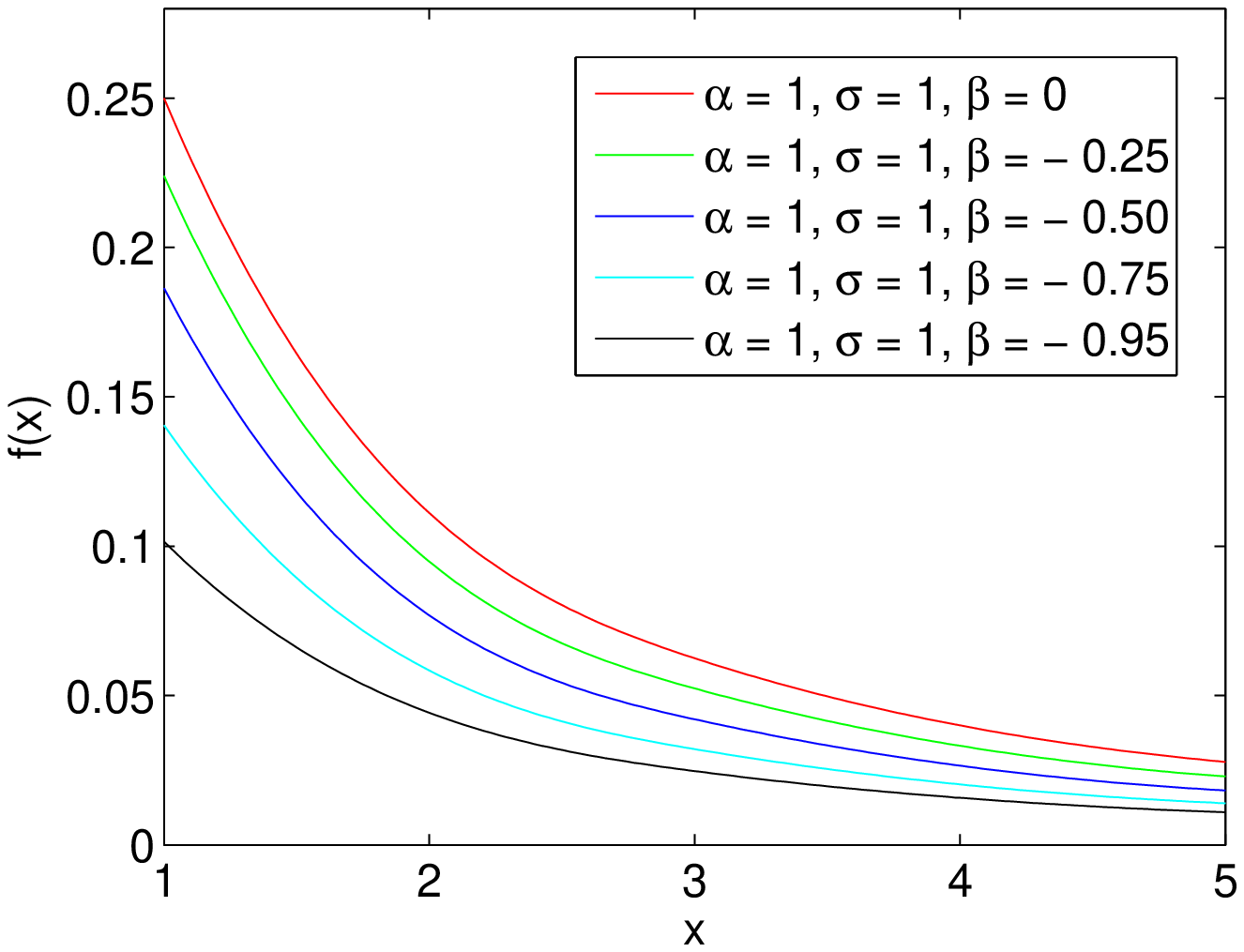}
\end{subfigure}
\begin{subfigure}{.5\textwidth}
  \centering
  \includegraphics[width=.9\linewidth]{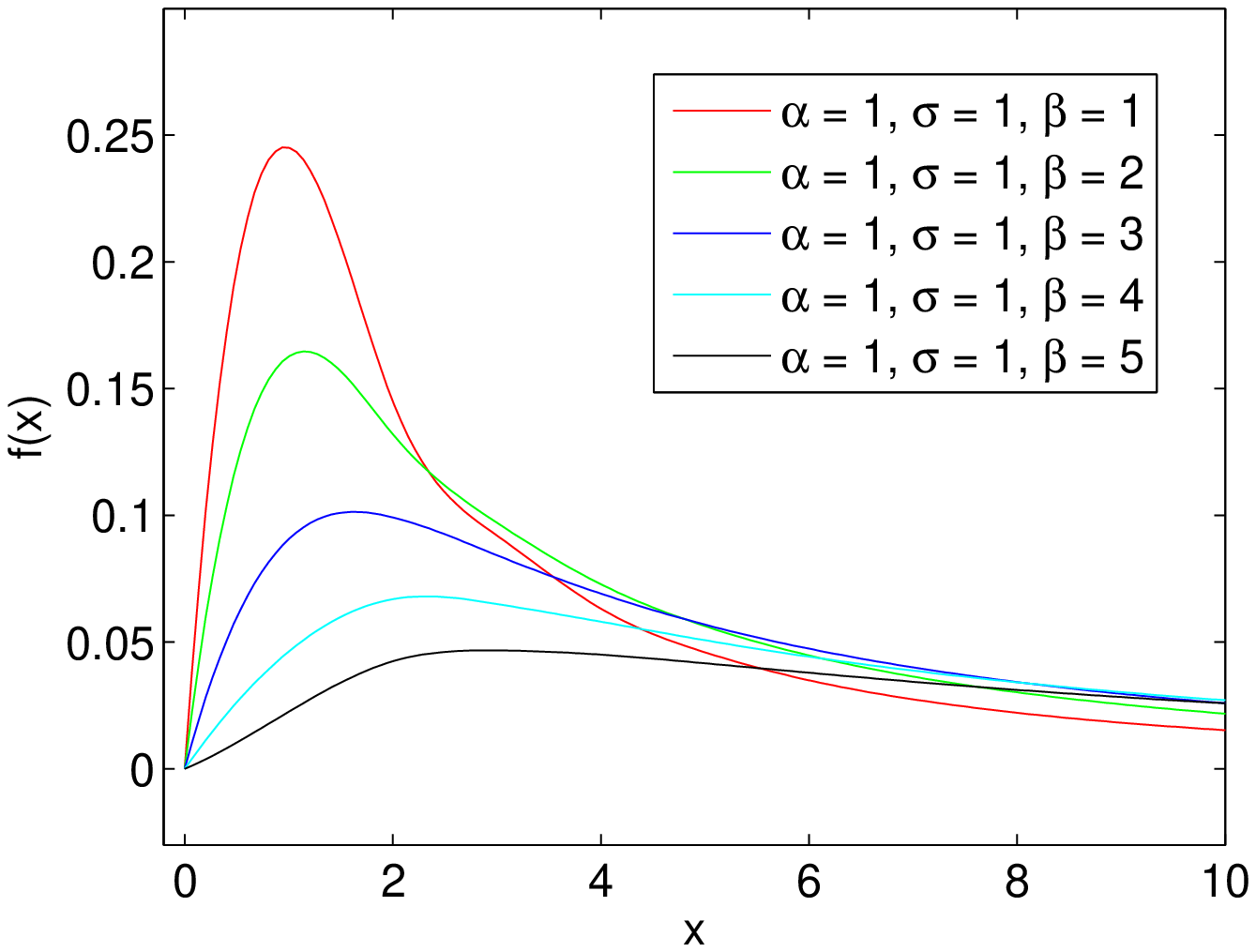}
\end{subfigure}
\begin{subfigure}{.5\textwidth}
  \centering
  \includegraphics[width=.9\linewidth]{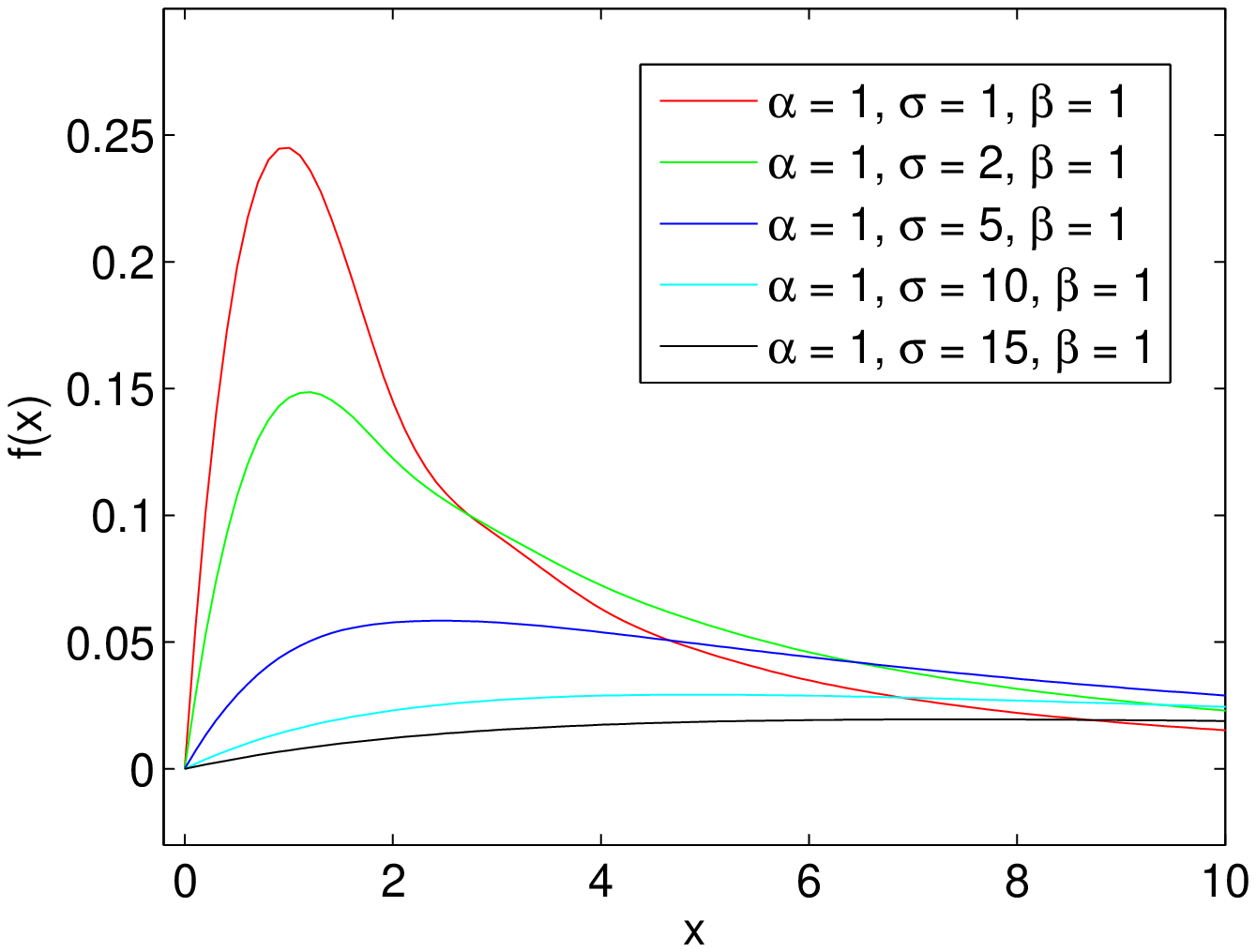}
\end{subfigure}
\begin{subfigure}{.5\textwidth}
  \centering
  \includegraphics[width=.9\linewidth]{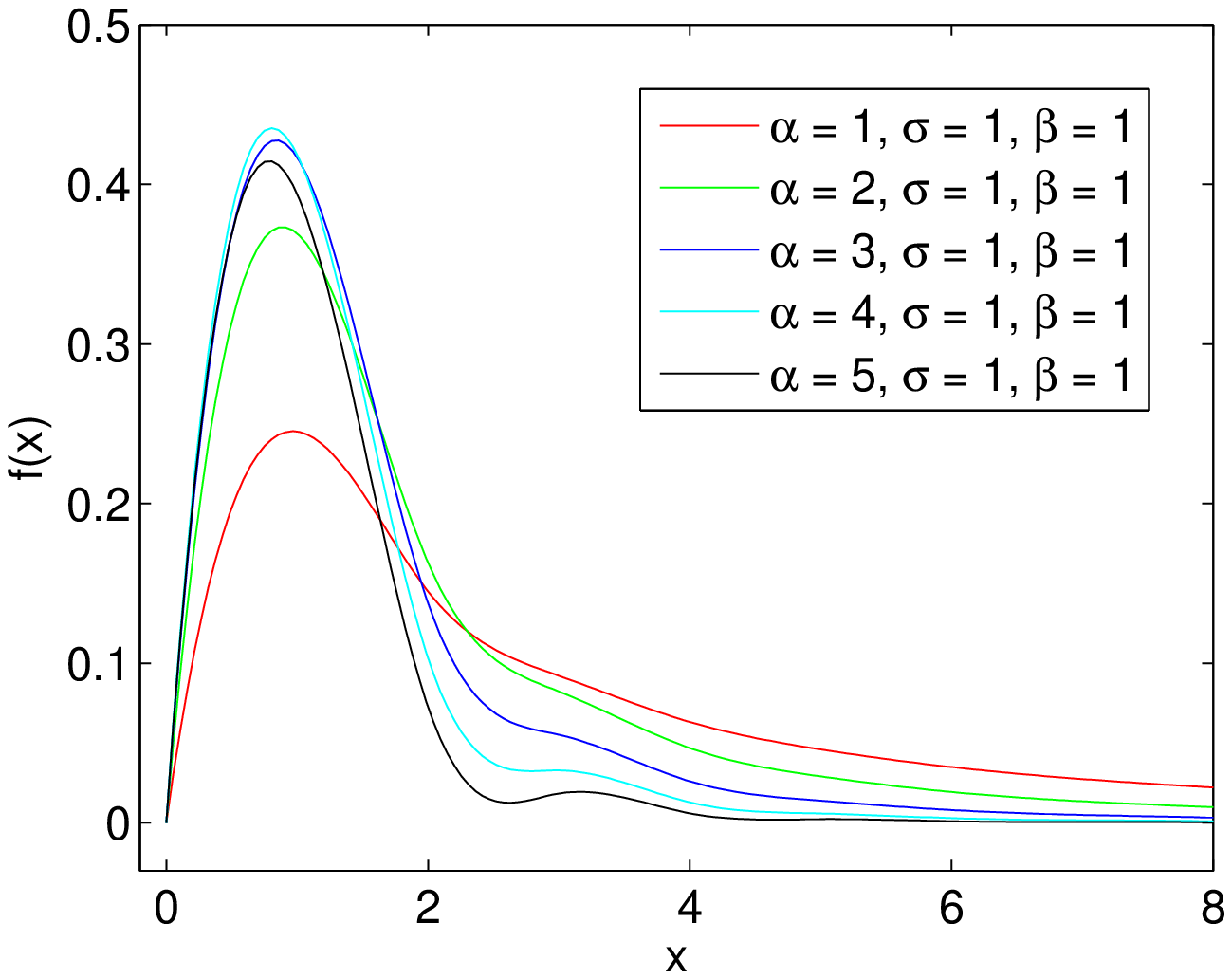}
\end{subfigure}
\caption{Plot of the PDFs of MLM distribution} \label{fig:pdf_plot}
\end{figure}

\begin{theorem}
Let $X$ be the random variable follows $MLM(\alpha, \beta, \sigma)$
distribution, then the PDF as in (\ref{gplpdf}) is a decreasing
function for $-1< \beta <0$. \label{mono}
\end{theorem}
\begin{proof}
Differentiating (\ref{gplpdf}) w.r.t. $x$, we have
\begin{equation}
\begin{split}\label{derglm}\small
f'(x) & =
-\displaystyle\frac{\alpha^2\left[1-F(x)\right]\left[\beta+1+\log\left(1+\frac{x}{\sigma}\right)\right]^2
\left[\log\left(1+\frac{x}{\sigma}\right)\right]^{2\beta}}
{\sigma^2\left(1+\frac{x}{\sigma}\right)^2
\left[1+\log\left(1+\frac{x}{\sigma}\right)\right]^{2\beta+2}} \\
&
-\displaystyle\frac{\alpha\left[1-F(x)\right]\left\{\left[\beta+1+\log\left(1+\frac{x}{\sigma}\right)\right]
\left[\log\left(1+\frac{x}{\sigma}\right)\right]^{\beta}+(1+\beta)\left[\log\left(1+\frac{x}{\sigma}\right)\right]^{\beta-1}\right\}}
{\sigma^2\left(1+\frac{x}{\sigma}\right)^2
\left[1+\log\left(1+\frac{x}{\sigma}\right)\right]^{\beta+1}}\\
& +
\displaystyle\frac{\alpha\left[1-F(x)\right](1+\beta)\left[\beta+1+\log\left(1+\frac{x}{\sigma}\right)\right]
\left[\log\left(1+\frac{x}{\sigma}\right)\right]^{\beta-1}}
{\sigma^2\left(1+\frac{x}{\sigma}\right)^2
\left[1+\log\left(1+\frac{x}{\sigma}\right)\right]^{\beta+2}}
\end{split}
\end{equation}
Trivially, if $-1<\beta<0$, then $f'(x)< 0$. Thus, $f(x)$ is
decreasing function if $\beta \in (-1,0)$.
\end{proof}

\section{Statistical Properties of the MLM distribution} \label{statistics_prop}
\subsection{Characterization and existence of the likelihood}\label{existnce_mle_mlm}
Initially we characterize the  maximum likelihood estimates (MLEs)
of the parameters $\alpha$ and $\sigma$ of a Lomax distribution.
Subsequently, we derived a sufficient condition for the existence of
MLEs of the MLM distribution using coefficient of variation (CV).
Given a set of samples $\{x_i\}$ of size $n$, the log-likelihood
function for the Lomax distribution, after dividing it by
the sample size $n$, is given by
\begin{equation}
\ell(\alpha, \sigma) = log{\alpha}- log{\sigma} -
\frac{(\alpha+1)}{n}\sum_{\substack{i=1}}^{n}\log\left(1+\frac{x_i}{\sigma}\right)
\label{4}
\end{equation}
Differentiating (\ref{4}) w.r.t. $\alpha$ and $\sigma$,
respectively, we have:
\begin{equation}
\frac{\partial\ell(\alpha, \sigma)}{\partial\alpha} =
\frac{1}{\alpha} -
\frac{1}{n}\sum_{\substack{i=1}}^{n}\log\left(1+\frac{x_i}{\sigma}\right)
\label{5}
\end{equation}
\begin{equation}
\frac{\partial\ell(\alpha, \sigma)}{\partial\sigma} = -
\frac{1}{\sigma} +
\frac{(1+\alpha)}{n\sigma}\sum_{\substack{i=1}}^{n}\left(\frac{x_i}{\sigma
+ x_i}\right) \label{6}
\end{equation}
Equating to zero the derivative of $\ell(\alpha, \sigma)$ w.r.t.
$\alpha$ in (\ref{5}), we obtain $\hat{\alpha}=\alpha(\sigma)$ as
follows:
\begin{equation}
\hat{\alpha}=\alpha(\sigma)=
\frac{n}{\sum_{\substack{i=1}}^{n}\log\left(1+\frac{x_i}{\sigma}\right)}\label{7}
\end{equation}
Differentiating (\ref{7}) w.r.t. $\sigma$ we have,
\begin{equation}
\alpha'(\sigma)=
\frac{\hat{\alpha}^2}{n\sigma}\sum_{\substack{i=1}}^{n}\frac{x_i}{\sigma+x_i}\label{8}
\end{equation}
It is important to note that there is no closed form solution to the
likelihood based on (\ref{5}) and (\ref{6}), and a suitable
numerical algorithm (for example, Newton-Raphson method) can be
employed to obtain the maximum likelihood estimates (MLEs) of the
$\alpha$ and $\sigma$. Different estimation procedures of the MLEs
have been discussed in previous literature, for example see
\cite{giles2013bias}. But for small or medium-sized samples,
anomalous behavior of the likelihood surface can be encountered when
sampling from the Lomax distribution. In this paper, we
characterize the profile log-likelihood function in terms of the
coefficient of variation (CV), defined as follows:
\begin{definition}
The CV is the ratio of the standard deviation $(s)$ to the mean
$(\mu)$, $$CV = \frac{s}{\mu};$$ where $\mu =
\frac{1}{n}\sum_{\substack{i=1}}^{n}x_i$ and $s=
\sqrt{\frac{1}{n}\sum_{\substack{i=1}}^{n}x_i^{2} - \mu^2}$.
\end{definition}
By using standard notation, the profile log-likelihood function based on equation \ref{4}, is given by
\begin{equation}
\ell_p(\sigma)=\sup \ell(\hat{\alpha}, \sigma)= log(\alpha(\sigma))
- log{\sigma} - 1 - \frac{1}{\alpha(\sigma)} \label{9}
\end{equation}
Differentiating (\ref{9}) w.r.t. $\sigma$, we have the following:
\begin{equation}
\ell_{p}^{'}(\sigma)= \frac{\alpha^{'}(\sigma)}{\alpha(\sigma)} -
\frac{1}{\sigma} +
\frac{\alpha^{'}(\sigma)}{\big[\alpha(\sigma)\big]^2} \label{10}
\end{equation}
Below we present the following lemmas which will be useful to find
the sufficient condition for the existence for the global maximum of
the profile log-likelihood function (\ref{9}).
\begin{lemma}
The following limit holds:
\begin{enumerate}
\item $\displaystyle \lim_{\sigma \to \infty} \sigma
\log\left(1+\frac{x}{\sigma}\right)=x;$
\item $\displaystyle \lim_{\sigma \to \infty} \frac{\sigma x}{\sigma + x}=x;$
\item $\displaystyle \lim_{\sigma \to \infty} \sigma^2 \left(
\log\left(1+\frac{x}{\sigma}\right) - \frac{x}{\sigma +
x}\right)=\frac{x^2}{2}.$
\end{enumerate}\label{lem1}
\end{lemma}
\begin{proof}
The proof is elementary and can easily be done using series
expansions.
\end{proof}
\begin{lemma}
The following limit holds:
\begin{enumerate}
\item $\displaystyle \lim_{\sigma \to \infty} \frac{1}{\alpha(\sigma)}=0;$
\item $\displaystyle \lim_{\sigma \to \infty}
\frac{\alpha(\sigma)}{\sigma}=\frac{1}{\bar{x}},$ where $\bar{x}$ is
the sample mean;
\item $\ell_{0} \equiv \displaystyle \lim_{\sigma \to \infty} \ell_{p}(\sigma) = \log\left(\frac{1}{\bar{x}}\right) - 1.$
\end{enumerate}\label{lem2}
\end{lemma}
\begin{proof}
The proofs are straightforward and can be done using Lemma
(\ref{lem1}).
\begin{enumerate}
\item $\displaystyle \lim_{\sigma \to \infty} \frac{1}{\alpha(\sigma)}=\displaystyle \lim_{\sigma \to \infty} \frac{1}{n}{\sum_{\substack{i=1}}^{n}\log\left(1+\frac{x_i}{\sigma}\right)} = \displaystyle \lim_{\sigma \to \infty} O\left(\frac{1}{\sigma}\right)=0.$
\item $\displaystyle \lim_{\sigma \to \infty}
\frac{\alpha(\sigma)}{\sigma} = \displaystyle \lim_{\sigma \to
\infty} \frac{n}{\sigma
\sum_{\substack{i=1}}^{n}\log\left(1+\frac{x_i}{\sigma}\right)} =
\frac{n}{\sum_{\substack{i=1}}^{n}{x_i}}=\frac{1}{\bar{x}}.$
\item $\displaystyle \lim_{\sigma \to \infty} \ell_{p}(\sigma) = \displaystyle \lim_{\sigma \to \infty} \left[ \log\left(\frac{\alpha(\sigma)}{\sigma}\right) - 1 -
\frac{1}{\alpha(\sigma)} \right] =\log\left(\frac{1}{\bar{x}}\right)
- 1.$
\end{enumerate}
\end{proof}

A sufficient condition for monotonic increasing (decreasing) for the
profile log-likelihood function is presented in Theorem
(\ref{theo2}) below, for sufficiently large $\sigma$. Also, we
present a sufficient condition for the existence of global maximum corresponding to the likelihood function for the Lomax distribution to be at a
finite point in Corollary (\ref{cor1}).
\begin{theorem}
Let $X$ follows $LM(\alpha, \sigma)$ distribution with $\alpha,
\sigma > 0$. A sufficient condition for $\ell_{p}(\sigma)$ to be
monotonically decreasing function is CV$ > 1$ for $\sigma \to
\infty$, and if CV$ < 1$, it is monotonically increasing.
\label{theo2}
\end{theorem}
\begin{proof}
Using (\ref{7}) and (\ref{8}) in Eqn. (\ref{10}), we can write
$\ell_{p}^{'}(\sigma)$ as:
\begin{equation}
\ell_{p}^{'}(\sigma)= - \frac{1}{\sigma} \left[
\frac{\sum_{\substack{i=1}}^{n}\log\left(1+\frac{x_i}{\sigma}\right)-
\sum_{\substack{i=1}}^{n} \frac{x_i}{\sigma +
x_i}}{\sum_{\substack{i=1}}^{n}\log\left(1+\frac{x_i}{\sigma}\right)}
\right] + \frac{1}{n\sigma}\sum_{\substack{i=1}}^{n}
\frac{x_i}{\sigma + x_i} \label{11}
\end{equation}
Using the limits of Lemma (\ref{lem1}) in Eqn. (\ref{11}), we have
\begin{equation}
- \displaystyle \lim_{\sigma \to \infty} \sigma^2
\ell_{p}^{'}(\sigma)= \frac{1}{2} \times
\frac{\sum_{\substack{i=1}}^{n} x_i^2}{\sum_{\substack{i=1}}^{n}
x_i} - \bar{x}. \label{12}
\end{equation}
Finally, we note that $- \displaystyle \lim_{\sigma \to \infty}
\sigma^2 \ell_{p}^{'}(\sigma) > 0 $ when the R.H.S of Eqn.(\ref{12})
is strictly greater than 0. Alternatively, the likelihood function
is monotonic decreasing when $\frac{1}{2n}\sum_{\substack{i=1}}^{n}
x_i^2 - \bar{x}^2 > 0$, or, equivalently, CV$ > 1.$ In a similar
way, we can show that if CV$ < 1$, then the $\ell_{p}(\sigma)$ is
monotonic increasing function for sufficiently large $\sigma$.
\end{proof}
\begin{remark}
As a consequence of Theorem (\ref{theo2}), it can be immediately
concluded that $\ell_{p}(\sigma)$ tends to $\ell_{0}$ based on Lemma
(\ref{lem2}) and $\ell_{p}(\sigma)$ is a monotonic function for
sufficiently large $\sigma$. The value of CV as a measure that can
be useful to determine when $\ell_{p}(\sigma)$ will be monotonic
increasing or decreasing function for sufficiently large $\sigma$.
\end{remark}
\begin{corollary}
Given a set of samples $\{x_i\}$ of $(+)ve$ numbers with CV$ > 1$,
the profile likelihood function for the $LM(\alpha, \sigma)$
distribution has a global maximum at a finite point.\label{cor1}
\end{corollary}
\begin{proof}
For small or moderate values of $\sigma$, using (\ref{7}), we have
\begin{equation}
\displaystyle \lim_{\sigma \to 0}\alpha(\sigma)= \displaystyle
\lim_{\sigma \to
0}\frac{n}{\sum_{\substack{i=1}}^{n}\log\left(1+\frac{x_i}{\sigma}\right)}=0.\label{13}
\end{equation}
Now, using (\ref{13}) in (\ref{9}) we have the following:
\begin{equation}
\displaystyle \lim_{\sigma \to 0} \ell_{p}(\sigma)= -
\infty.\label{14}
\end{equation}
Since $ \ell_{p}(\sigma)$ is a continuous and monotonic decreasing
function for sufficiently large $\sigma$ (as in Theorem \ref{theo2})
and using (\ref{14}), we can conclude that a global maximum exists
at a finite point when CV$ > 1$.
\end{proof}
\begin{remark}
Corollary \ref{cor1} shows that the likelihood function for the
Lomax distribution has a global maximum for the samples $\{x_i\}$
with CV$ > 1$ at a finite point. The calculation of CV is completely
based on available empirical data and easy to compute. The existence
of MLE based on CV for the MLM distribution will also holds as
because MLM model reduce to Lomax distribution when $ \displaystyle
\lim_{x \to \infty} m(x) = \alpha$. This can be empirically
validated in section \ref{experiments}\ref{analysis_results} and
will be useful  useful from practitioner's point of view.
\end{remark}

\subsection{MLE of parameters}
In this section, the maximum likelihood estimates are derived for
parameters $\alpha, \beta, \; \mbox{and} \; \sigma$ of MLM
distribution. Let  $x_1, x_2, ... , x_n$ be a sample of size $n$
from MLM($\alpha,\beta,\sigma$) distribution. Then the
log-likelihood function for the vector of parameters
$\Theta=(\alpha,\beta,\sigma)^{T}$ is given by
\begin{equation}
\label{eq1}
\begin{split}
\ell\equiv\ell(x;\alpha,\beta,\sigma)
&=n\log(\alpha)-\sum_{\substack{i=1}}^{n}\log\left(\sigma+x_i\right)+\sum_{\substack{i=1}}^{n}\log\left[\beta+1+\log\left(1+\frac{x_i}{\sigma}\right)\right]\\
&+\beta\sum_{\substack{i=1}}^{n}\log\left[\log\left(1+\frac{x_i}{\sigma}\right)\right]-(\beta+1)\sum_{\substack{i=1}}^{n}\log\left[1+\log\left(1+\frac{x_i}{\sigma}\right)\right] \\
& -\alpha\sum_{\substack{i=1}}^{n}
\displaystyle\frac{[\log\left(1+\frac{x_i}{\sigma}\right)]^{\beta+1}}{[1+\log\left(1+\frac{x_i}{\sigma}\right)]^\beta},\\
\end{split}
\end{equation}

The maximum likelihood estimate for the parameters $\alpha$,$\beta$,
and $\sigma$ are given by $\hat{\alpha}$,$\hat{\beta}$, and
$\hat{\sigma}$, are obtained by maximizing the likelihood function
in Equation (\ref{eq1}). The first-order partial derivatives of (1)
with respect to $\alpha$, $\beta$, and $\sigma$ are

\begin{equation}
\label{eq2} \frac{\partial\ell}{\partial\alpha}=\frac{n}{\alpha} -
\sum_{i=1}^{n}\frac{\left[\log\left(1+\frac{x_i}{\sigma}\right)\right]^{\beta+1}}{\left[1+
\log\left(1+\frac{x_i}{\sigma}\right) \right]^{\beta}}
\end{equation}

\begin{equation}
\label{eq3}
\frac{\partial\ell}{\partial\beta}=\sum_{i=1}^{n}\frac{1}{\left(1+\beta+w_i
\right)} + \sum_{i=1}^{n}\log\left( \frac{w_i}{1+w_i} \right) \times
\left[ 1-\frac{\alpha w_i^{\beta+1}}{\left( 1+w_i \right)^{\beta}}
\right]
\end{equation}

\begin{equation}
\label{eq4} \frac{\partial\ell}{\partial\sigma}
=-\sum_{i=1}^{n}\frac{1}{\left( \sigma + x_i \right)} +
\sum_{i=1}^{n} \frac{x_i}{\sigma (\sigma + x_i)} \left[ \frac{\beta
+1}{(1+w_i)} - \frac{\beta}{w_i} - \frac{1}{(1+\beta +w_i)} \right]
+ \alpha \sum_{i=1}^{n}\frac{x_i}{\sigma(\sigma + x_i)} \left
[\frac{(1+\beta + w_i)w_i^{\beta}}{(1+w_i)^{\beta+1}} \right ],
\end{equation}

where $w_i=\log\left( 1 + \frac{x_i}{\sigma} \right )$.

The MLEs of the three parameters of the MLM($\alpha,\beta,\sigma$)
distributions are obtained by setting these above equations to zero
and solving them simultaneously. Closed forms of the solutions are
not available for the equations (\ref{eq2}), (\ref{eq3}) and
(\ref{eq4}). So, iterative methods will be applied to solve these
equations numerically.

\subsection{Asymptotic distribution}

Fisher information matrix, a measure of the information content of
the data relative to the parameters to be estimated, plays an
important role in parameter estimation. The Fisher information
matrix $(F)$ can be obtained by taking the expected values of the
second-order and mixed partial derivatives of
$\ell(\alpha,\beta,\sigma)$ w.r.t. $\alpha$, $\beta$, and $\sigma$.
Since, the analytical expression is hard to compute. Thus, it can be
approximated by numerically investing the the $F=(F_{ij})$ matrix.
The asymptotic $F$ matrix can be given as follows:
\[
F=\begin{bmatrix}
-\frac{\partial^{2}\ell}{\partial\alpha^{2}} & -\frac{\partial^{2}\ell}{\partial\alpha \partial\beta} & -\frac{\partial^{2}\ell}{\partial\alpha \partial \sigma} \\
-\frac{\partial^{2}\ell}{\partial\alpha \partial \beta} & -\frac{\partial^{2}\ell}{\partial\beta^{2}} & -\frac{\partial^{2}\ell}{\partial\beta \partial \sigma} \\
-\frac{\partial^{2}\ell}{\partial\alpha \partial \sigma} & -\frac{\partial^{2}\ell}{\partial\beta \partial \sigma} & -\frac{\partial^{2}\ell}{\partial \sigma^{2}} \\
\end{bmatrix}
\]
The second and mixed partial derivatives of the log likelihood
function are obtained as follows:
\begin{equation}
\label{eq5} \frac{\partial^{2}\ell}{\partial\alpha^{2}} =
-\frac{n}{\alpha^2}
\end{equation}

\begin{equation}
\label{eq6} \frac{\partial^{2}\ell}{\partial\alpha
\partial\beta}=\sum_{i=1}^{n}\log \left( \frac{1+w_i}{w_i} \right)
\times \left [ \frac{w_i^{\beta+1}}{(1+w_i)^{\beta}} \right ]
\end{equation}

\begin{equation}
\label{eq7} \frac{\partial^{2}\ell}{\partial\alpha \partial
\sigma}=\sum_{i=1}^{n} \frac{x_i}{\sigma (\sigma +x_i)} \times
\left[ \frac{w_i^{\beta}(1+\beta + w_i)}{(1+w_i)^{\beta+1}} \right ]
\end{equation}

\begin{equation}
\label{eq8}
\frac{\partial^{2}\ell}{\partial\beta^{2}}=-\sum_{i=1}^{n}\frac{1}{(1+\beta+w_i)^{2}}-
\alpha \sum_{i=1}^{n} \log^2 \left( \frac{w_i}{1+w_i} \right ) \left
[ \frac{w_i^{\beta+1}}{(1+w_i)^{\beta}} \right ]
\end{equation}

\begin{equation}
\label{eq9}
\begin{split}
\frac{\partial^{2}\ell}{\partial\beta \partial \sigma}
& = \sum_{i=1}^{n} \frac{x_i}{\sigma (\sigma +x_i) (1+\beta + w_i)^{2}} -\sum_{i=1}^{n} \frac{x_i}{\sigma (\sigma +x_i)} \times \left [ \frac{(1+w_i)^{\beta} - \alpha w_i^{\beta+1}}{w_i (1+w_i)^{\beta+1}} \right ] \\
& + \alpha \sum_{i=1}^{n} \frac{x_i}{\sigma (\sigma + x_i)} \left [
\frac{(1+\beta +w_i)w_i^{\beta}}{(1+w_i)^{\beta+1}} \right ] \left [
\log \left ( \frac{w_i}{1+w_i} \right ) \right ]
\end{split}
\end{equation}

\begin{equation}
\label{eq10}
\begin{split}
\frac{\partial^{2}\ell}{\partial \sigma^{2}}
& = \sum_{i=1}^{n}\frac{1}{(\sigma + x_i)^{2}} + \sum_{i=1}^{n} \frac{x_i^{2}}{\sigma^{2}(\sigma + x_i)^2} \left [ \frac{\beta+1}{(1+w_i)^2} -\frac{\beta}{w_i^{2}} -\frac{1}{(1+\beta + w_i)^2} \right ] \\
& + \sum_{i=1}^{n} \frac{x_i(2\sigma + x_i)}{\sigma^{2} (\sigma +x_i)^2} \left [ \frac{\beta}{w_i} + \frac{1}{(1+\beta + w_i)} - \frac{\beta +1}{(1+w_i)} \right ] -\alpha  \sum_{i=1}^{n} \frac{x_i^{2}(1+\beta)}{\sigma^{2}(x_i+\sigma)^2} \left [  \frac{w_i^{\beta -1}(\beta + w_i)}{(1+w_i)^{\beta+1}} \right ] \\
& -\alpha \sum_{i=1}^{n}\frac{x_i}{\sigma^{2}(\sigma +x_i)^2} \left
[ \frac{w_i^{\beta} (1+\beta + w_i) \left [ (2\sigma +x_i) (1+w_i) -
x_i(\beta +1) \right ]}{(1+w_i)^{\beta +2}} \right ]
\end{split}
\end{equation}

The variance-covariance matrix is approximated by $M = (M_{ij})$
where $M_{ij}=F_{ij}^{-1}$. The asymptotic distribution of MLEs for
$\alpha$, $\beta$, and $\sigma$ can be written as
\[
\left [  (\hat{\alpha}-\alpha), (\hat{\beta}-\beta), (\hat{\sigma} -
\sigma) \right ] \sim N_3 (0,F^{-1}(\hat{\theta}))
\]
Then the approximate $100(1-k)\%$ confidence intervals for $\alpha$,
$\beta$, and $\sigma$ are given by $\hat{\alpha} \pm
\mathscr{Z}_{\frac{k}{2}} \sqrt{Var(\hat{\alpha})}$, $\hat{\beta}
\pm \mathscr{Z}_{\frac{k}{2}} \sqrt{Var(\hat{\beta})}$, and
$\hat{\sigma} \pm \mathscr{Z}_{\frac{k}{2}}
\sqrt{Var(\hat{\sigma})}$, where
$\hat{\Theta}=(\hat{\alpha},\hat{\beta},\hat{\sigma})$ and
$\mathscr{Z}_k$ is the upper 100 k-th percentile of the standard
normal distribution.

\subsection{Extreme value properties}

Here we study some of the interesting extreme value theoretic
properties. The concept of regular variation is an important notion
of extreme value theory. Below we show the extreme value results for the MLM distribution that can characterize the asymptotic behavior of extremes along with well grounded statistical theory.

\begin{definition}{(Maximum domain of attraction)}
A function $F$ is said to be regularly varying at infinity, if for
every $t>0,$
$$ \displaystyle \lim_{x \to \infty}
\frac{1-F(tx)}{1-F(x)}=t^{-\alpha}; \; \alpha>0.$$ Then we say that
$F$ is a function with regularly varying tails with $\alpha >0$ as
the tail index and $F$ belongs to the maximum domain of attraction
(MDA) of the Frechet distribution with index $\alpha$.
\end{definition}
\begin{theorem}
The CDF (Eqn. \ref{gplcdf}) of the $MLM$ distribution is a function
with regularly varying tails and it belongs to MDA of the Frechet
distribution with index $\alpha$.
\end{theorem}
\begin{proof}
\begin{equation}
1 - F(tx) = \exp \Bigg[ - \alpha . \frac{\log ^ {\beta +1} \big( 1 +
\frac{tx}{\sigma} \big)}{\bigg( 1+\log \big( 1 + \frac{tx}{\sigma}
\big) \bigg)^{\beta}} \Bigg] ; \; t>0 \label{F(tx)}
\end{equation}
Now, we have (using expansions of $\log(1-x)$ and $\exp(x)$):
\begin{align}
\Bigg( \frac{\log \big( 1 + \frac{tx}{\sigma} \big)}{1+\log \big( 1
+ \frac{x}{\sigma} \big)} \Bigg)^{\beta} & = \Bigg( 1
-\frac{1}{1+\log \big( 1 + \frac{tx}{\sigma} \big)} \Bigg)^{\beta}
\nonumber\\
& = \exp \Bigg[ \beta \log\Bigg( 1 -\frac{1}{1+\log \big( 1 +
\frac{tx}{\sigma} \big)} \Bigg) \Bigg] \nonumber\\
& = \exp \Bigg[ \beta \Bigg( - \frac{1}{\log \big( 1 +
\frac{tx}{\sigma} \big)} + O \bigg( \frac{1}{\log^2 \big( 1 +
\frac{tx}{\sigma} \big)} \bigg) \Bigg) \Bigg] \nonumber\\
& = 1 - \frac{\beta}{\log \big( 1 + \frac{tx}{\sigma} \big)} + O
\bigg( \frac{\beta}{\log^2 \big( 1 + \frac{tx}{\sigma} \big)} \bigg)
\label{expansion}
\end{align}
Using Eqn. (\ref{F(tx)}) and (\ref{expansion}) together, we get
\begin{equation}
1 - F(tx) = \exp \Bigg[ -\alpha \log \bigg( 1 + \frac{tx}{\sigma}
\bigg) \Bigg\{ 1 - \frac{\beta}{\log \big( 1 + \frac{tx}{\sigma}
\big)} + O \bigg( \frac{\beta}{\log^2 \big( 1 + \frac{tx}{\sigma}
\big)} \bigg) \Bigg\} \Bigg] \label{Ftx}
\end{equation}
Similarly for $t=1$, Eqn. (\ref{Ftx}) becomes
\begin{equation}
1 - F(x) = \exp \Bigg[ -\alpha \log \bigg( 1 + \frac{x}{\sigma}
\bigg) \Bigg\{ 1 - \frac{\beta}{\log \big( 1 + \frac{x}{\sigma}
\big)} + O \bigg( \frac{\beta}{\log^2 \big( 1 + \frac{x}{\sigma}
\big)} \bigg) \Bigg\} \Bigg] \label{Fx}
\end{equation}

Now,
\begin{equation*}
\begin{split}
\displaystyle \lim_{x \to \infty} \frac{1-F(tx)}{1-F(x)} & =
\displaystyle \lim_{x \to \infty} \exp\Bigg[ - \alpha \log \Bigg(
\frac{1 + \frac{tx}{\sigma}}{1 + \frac{x}{\sigma}} \Bigg) + O\Bigg(
\frac{1}{\log^2 \big( 1 + \frac{tx}{\sigma} \big)} + \frac{1}{\log^2
\big( 1 + \frac{x}{\sigma} \big)} \Bigg) \Bigg] \\
& = \exp\big(-\alpha\log t \big) \\
& = t^{-\alpha}.
\end{split}
\end{equation*}
Thus, $F \in  MDA(\Phi_\alpha)$.
\end{proof}
Now we study the tail-equivalent and heavy-tailed behaviour of the
proposed MLM distribution as follows:

\begin{definition}{(Tail-equivalent)}
Two distributions $F$ and $G$ are said to be tail-equivalent if
$$ \displaystyle \lim_{x \to \infty} \frac{1-F(x)}{1-G(x)}=c ; \;
0<c<\infty.$$

\end{definition}
\begin{theorem}
The $MLM (\alpha, \beta, \sigma)$ distribution, defined in Eqn.
(\ref{gplcdf}), is right tail-equivalent to the power-law
distribution.
\end{theorem}
\begin{proof}
Let $G(x)$ be the CDF of the power-law distribution, i.e., $$ 1-
G(x) = \bigg( 1 + \frac{x}{\sigma} \bigg)^{-\alpha} $$ and $F(x)$ is
the CDF of MLM distribution as given in Eqn.(\ref{gplcdf}). Then,
\begin{equation*}
\begin{split}
\displaystyle \lim_{x \to \infty} \frac{1-F(x)}{1-G(x)} & =
\displaystyle \lim_{x \to \infty} \frac{\exp\Bigg[ -\alpha \log
\big( 1 + \frac{x}{\sigma} \big) + \alpha\beta + O\Bigg(
\frac{1}{\log \big( 1 + \frac{x}{\sigma} \big)} \Bigg) \Bigg]}
                                           {\exp\Bigg[ -\alpha \log \big( 1 + \frac{x}{\sigma} \big) \Bigg]} \; \mbox{(Using Eqn. (\ref{Fx}))} \\
& = \displaystyle \lim_{x \to \infty} \exp\Bigg[ \alpha\beta +
O\Bigg(\frac{1}{\log \big( 1 + \frac{x}{\sigma} \big)} \Bigg) \Bigg] \\
& = \exp\big(\alpha\beta \big) \\
& = c < \infty.
\end{split}
\end{equation*}
\end{proof}
\begin{definition}{(Heavy-tailed distribution)}
A distribution function $F$ is heavy-tailed if
$$ \displaystyle \lim_{x \to \infty} \exp\{\lambda x\}\big(1-F(x)\big)
= \infty, \; \mbox{for any} \; \lambda > 0.$$
\end{definition}
\begin{theorem}
The $MLM (\alpha, \beta, \sigma)$ distributions, defined in Eqn.
(\ref{gplcdf}), are heavy-tailed distributions.
\end{theorem}
\begin{proof}
\begin{equation*}
\begin{split}
\displaystyle \lim_{x \to \infty} \exp\{\lambda x\}\big(1-F(x)\big)
& = \lim_{x \to \infty} {\exp\Bigg[ \lambda x -\alpha \log \bigg( 1
+ \frac{x}{\sigma} \bigg) + \alpha\beta + O\Bigg( \frac{1}{\log
\big( 1
+ \frac{x}{\sigma} \big)} \Bigg) \Bigg]} \\
& = \infty ,
\end{split}
\end{equation*}
since $\log \big( 1 + \frac{x}{\sigma} \big) \asymp x^{\epsilon}$
for any $\epsilon > 0$ and for sufficiently large $x$.
\end{proof}

There are two other important class of distributions
\cite{embrechts2013modelling} viz. the class $\mathbb{D}$ of
dominated-variation distributions and and the class $\mathbb{L}$ of
long-tailed distributions that are used in the risk theory and
queueing theory. The proposed MLM distributions also follows these
two properties.

\begin{definition}
A distribution $F$ belong to the class $\mathbb{D}$ of
dominated-variation distributions if $$ \displaystyle \limsup
\limits_{x \to \infty} \frac{1-F(x)}{1-F(2x)}< \infty.
$$
\end{definition}
\begin{theorem}\label{classD}
If $\alpha > 0$, then $MLM (\alpha, \beta, \sigma)$ distribution,
defined in Eqn. (\ref{gplcdf}), belongs to the class $\mathbb{D}$ of
dominated-variation distributions.
\end{theorem}
\begin{proof}
\begin{equation*}
\begin{split}
\displaystyle \lim_{x \to \infty} \frac{1-F(x)}{1-F(2x)} & =
\displaystyle \lim_{x \to \infty} \frac{\exp\Bigg[ -\alpha \log
\big( 1 + \frac{x}{\sigma} \big) + \alpha\beta + O\Bigg(
\frac{1}{\log \big( 1 + \frac{x}{\sigma} \big)} \Bigg) \Bigg]}
                                           {\exp\Bigg[ -\alpha \log \big( 1 + \frac{2x}{\sigma} \big) + \alpha\beta + O\Bigg( \frac{1}{\log \big( 1 + \frac{2x}{\sigma} \big)} \Bigg) \Bigg]} \\
& = \displaystyle \lim_{x \to \infty} \exp\Bigg[ \alpha \log \Bigg(
\frac{1 + \frac{2x}{\sigma}}{1 + \frac{x}{\sigma}} \Bigg) + O\Bigg(
\frac{1}{\log \big( 1 + \frac{x}{\sigma} \big)} + \frac{1}{\log
\big( 1 + \frac{2x}{\sigma} \big)} \Bigg) \Bigg] \\
& =  \exp\big(\alpha\log2 \big) \\
& = 2^{\alpha} < \infty.
\end{split}
\end{equation*}
where $\alpha > 0$.
\end{proof}

\begin{definition}
A distribution $F$ is said to belong to the class $\mathbb{L}$ of
long-tailed distributions if $F$ has right unbounded support and,
for any fixed $y > 0$,
$$ \displaystyle \lim_{x \to \infty}
\frac{1-F(x+y)}{1-F(x)}=1.$$
\end{definition}
\begin{theorem}\label{classL}
The $MLM (\alpha, \beta, \sigma)$ distribution, defined in Eqn.
(\ref{gplcdf}), belongs to the class $\mathbb{L}$ of long-tailed
distributions.
\end{theorem}
\begin{proof}
\begin{equation*}
\begin{split}
\displaystyle \lim_{x \to \infty} \frac{1-F(x+y)}{1-F(x)} & =
\displaystyle \lim_{x \to \infty} \exp\Bigg[- \alpha \log \Bigg( 1+
\frac{(y/\sigma)}{\big(1 + \frac{x}{\sigma}\big)} \Bigg) + O\Bigg(
\frac{1}{\log \big( 1 + \frac{x}{\sigma} \big)} + \frac{1}{\log
\big( 1 + \frac{x+y}{\sigma} \big)} \Bigg) \Bigg] \\
& = 1, \quad \text{since } \; \sigma > 0.
\end{split}
\end{equation*}
\end{proof}
We have shown that the proposed MLM distributions are heavy-tailed
and also possess the additional regularity property of
subexponentiality \cite{foss2011introduction} as given below.
Essentially this corresponds to good tail behaviour under the
operation of convolution.
\begin{definition}{(Subexponential distribution)}
We say that a distribution $F$ is subexponential if
$$ \displaystyle \lim_{x \to \infty} \frac{1-F\ast F(x)}{1-F(x)}=2,$$
where $\ast$ denotes the convolution operation.
\end{definition}
\begin{theorem}
The $MLM (\alpha, \beta, \sigma)$ distribution, defined in Eqn.
(\ref{gplcdf}), is subexponential.
\end{theorem}
\begin{proof}
Form Theorem \ref{classD} and Theorem \ref{classL}, the $MLM(\alpha,
\beta, \sigma)$ distribution belongs to $\mathbb{D}\cap\mathbb{L}$.
Using \cite{kluppelberg1988subexponential},
$\mathbb{D}\cap\mathbb{L} \subset \mathbb{S}$, where $S$ is the
class of subsexponetial distribution. Hence the theorem.
\end{proof}
\begin{definition}{(Von-Mises type function)}
A distribution function $F$ is called a Von-Mises type function if
$$ \displaystyle \lim_{x \uparrow r(F)} x\frac{d}{dx} \left[ \frac {1-F(x)}{xf(x)} \right]=0,$$
where $r(F)=\sup\{x: F(x)<1\}$ denotes the right extremity of the
distribution function $F$ \cite{embrechts2013modelling}.
\end{definition}
\begin{theorem}
The $MLM (\alpha, \beta, \sigma)$ distribution, defined in Eqn.
(\ref{gplcdf}), satisfies the Von-Mises condition.
\end{theorem}
\begin{proof}
\begin{equation*}
\begin{split}
\displaystyle \lim_{x \to \infty} x\frac{d}{dx} \left[ \frac
{1-F(x)}{xf(x)} \right] & = \frac{\alpha}{\sigma} \displaystyle
\lim_{x \to \infty} x \frac{d}{dx}\left[
\frac{ \left(1+\frac{x}{\sigma}\right)\left[1+\log\left(1+\frac{x}{\sigma}\right)\right]^{\beta+1}} {x \left[\beta+1+\log\left(1+\frac{x}{\sigma}\right)\right] \left[\log\left(1+\frac{x}{\sigma}\right)\right]^\beta} \right] \\
& = \frac{\alpha}{\sigma} \displaystyle \lim_{x \to \infty} \frac{\left[1+\log\left(1+\frac{x}{\sigma}\right)\right]^{\beta+1}}{ \left[\beta+1+\log\left(1+\frac{x}{\sigma}\right)\right] \left[\log\left(1+\frac{x}{\sigma}\right)\right]^\beta} \\
& \times \left[ - \frac{1}{x} + \frac{1+\beta}{\sigma \left(1+\log\left(1+\frac{x}{\sigma}\right)\right)} - \frac{1}{\sigma \left(1+\beta+\log\left(1+\frac{x}{\sigma}\right)\right)}- \frac{\beta}{\sigma \log\left(1+\frac{x}{\sigma}\right)}\right] \\
& = 0.
\end{split}
\end{equation*}
\end{proof}

\section{Experimental Analysis} \label{experiments}
\subsection{Description of data sets}
We present here the results of fitting modified Lomax (MLM)
distribution over 50 real-worlld complex networks
\cite{snapnets,rossi2015network} coming from broad variety of
different disciplines such as Social Networks, Collaboration
Networks, Communication Networks, Citation Networks, Temporal
Networks, Web Graphs, Product co-purchasing Networks, Biological
Networks, Brain Networks, etc. Please go through the supplementary
materials for more details about the data sets under consideration.
Some statistical measures of the data sets and the detailed
experimentation of the performances of the proposed MLM distribution
compared to the other common power-law related distribution such as
Lomax, Pareto, Log-normal, power-law cutoff, Exponential and Poisson
are discussed in the following sub sections.

\subsection{Performance measures}

Here we use some evaluation measures which justify that the degree distribution of a real-world complex network can plausibly been drawn from the proposed MLM distribution.  As  here  the  actual  distribution is  discrete, we can quantify the goodness-of-fit test (i.e., how closely a hypothesized distribution resembles the actual distribution) by calculating the Chi-square statistic value based on bootstrap resampling by generating 50000 synthetic data sets. The Chi-square test will return a $p$ value which quantifies the probability that our data were drawn from the hypothesized distribution. If the $p$ value is small (less than the significance level), we can reject the null hypothesis that the data come from the MLM distribution. We have also computed few other statistical measures such as KL-divergence, RMSE and MAE for quantifying the goodness-of-fit of the proposed MLM distribution model in comparison to the other standard distribution functions related to other heavy-tailed distributions.

\subsection{Analysis of results}\label{analysis_results}

\begin{table}[h]
\tiny
\centering \caption{Performance of the proposed MLM model over
different real-worlld networks}
\makebox[\textwidth]{%
    \begin{tabular}{ |c|c|c|c|c|c|c|c|c|c|c| }
        \hline
         \multicolumn{2}{|c|}{Data} & No. of  & No. of & \multicolumn{3}{c|}{Stat. Prop.}& \multicolumn{3}{c|}{Estimated} & Bootstrap\\
         \multicolumn{2}{|c|}{sets} & nodes & edges &  \multicolumn{3}{c|}{}  & \multicolumn{3}{c|}{parameters} & chi-square\\
         \multicolumn{2}{|c|}{} &  &  &  \multicolumn{3}{c|}{} & \multicolumn{3}{c|}{} &  value \\
         \multicolumn{2}{|c|}{} &  & & \multicolumn{1}{c}{$s$} & \multicolumn{1}{c}{$\mu$} & \multicolumn{1}{c|}{$\frac{s}{\mu}$} & \multicolumn{1}{c}{$\widehat{\alpha}$} & \multicolumn{1}{c}{$\widehat{\beta}$} & \multicolumn{1}{c|}{$\widehat{\sigma}$} & $(p)$\\
         \hline
         Social & ego-Twitter(In) & 81,306  & 1,768,149  & 57.965 & 21.747  & 2.6654 & 1.9922  & -0.3591  & 30.543  & 0.9920  \\
         Networks & ego-Gplus(In) & 107,614  & 13,673,453  & 1404.8& 283.42 & 4.9568 &  0.7108 &  -0.4983 & 23.077   & 0.9963 \\
         & soc-Slashdot & 70,068   & 358,647 & 35.069 & 10.237 & 3.426 & 0.8663  & -0.6228  & 1.0461  &  0.9955 \\
         & soc-Delicious(In) &  536,108  & 1,365,961 & 39.826 & 10.673 & 3.7312 & 1.3630  & -0.6819  & 5.3709   & 0.9960  \\
         & soc-Digg(In) &  770,799  & 5,907,132  & 166.61 & 46.584 & 3.5765 & 0.7931 &  -0.6928 & 5.5163  & 0.9890  \\
         & soc-Academia &  200,169  & 1,398,063  & 48.297 & 14.259 & 3.3871 & 2.7429 & -0.3737   & 36.644   & 0.6087  \\
         & LiveJournal(In) &  4,847,571  & 68,993,773  & 44.969 & 15.368 & 2.926 &  2.6892 & -0.7272   & 51.933  & 0.8983  \\
          & Dogster-Friendship & 426,821   & 8,546,581 & 284.06 & 40.033 & 7.095 & 1.5634 & 0.3108  & 14.057  & 0.9500  \\
          & Higgs-Twitter(In) & 456,626   & 14,855,842  & 350.91 & 54.786 & 6.4051 &  1.6797 &  -0.0347  & 36.204  & 0.9870  \\
         & Artist-Facebook & 50,615   & 819,307  & 63.427 & 32.366 & 1.9596 & 2.0117 &  -0.1445  & 39.337 & 0.9812  \\
         & Athletes-Facebook & 13,866   & 86,859 & 17.978 & 12.438 & 1.4453 & 3.1229 & 0.1406  & 21.180  & 0.9640  \\
         \hline
         Citation & cit-HepTh(In) & 27,770  & 352,807  & 43.139 & 15.220 & 2.8342 & 1.8410  & -0.3093  & 16.416  & 0.8730  \\
         Networks& cit-HepPh(In) & 34,546  & 421,578  & 27.286 & 14.933 & 1.8271 & 2.5553  & -0.3622  & 34.349   & 0.9900  \\
         & cit-Patents(In) & 3,774,768  & 16,518,948 & 6.9125 & 5.0687 & 1.3637 & 4.4822  & -0.2534  &  21.689 & 0.8080 \\
         & cit-Citeseer(In) & 227,320  & 814,134 & 9.8260 & 5.4322 & 1.8088 & 2.2630  &  -0.2788 & 7.4150  & 0.6350  \\
         \hline
         Collaboration & ca-CondMat & 23,133  & 93,497 & 10.671 & 8.0189 & 1.3308 &   3.1068 &  0.3615 & 10.5353   & 0.9896 \\
         Networks & ca-AstroPh & 18,772  & 198,110  & 30.568 & 21.103 & 1.4484 & 16.434  & 37.276   & 0.0101   & 0.9990  \\
         & ca-GrQc & 5,242  & 14,496  & 7.9186 & 5.5284 & 1.4322 & 2.2624  &  3.5861 & 0.6765  &  0.7849   \\
         & ca-HepPh & 12,008 &  118,521 & 46.654 & 19.696 & 2.3687 & 0.9798  &  2.8780 &  0.6791  &   0.8163  \\
         & ca-HepTh & 9,877 & 25,998 & 6.1867 & 5.2618 & 1.1757 & 2.9417  &  5.2791 &  0.4825 &   0.9332  \\
         \hline
         Web & Google(In) & 875,713  & 5,105,039  & 43.320 & 7.1444 & 6.0634 & 1.1999  &  -0.6399 & 2.0429   & 0.9780 \\
         Graphs & BerkStan(In) & 685,230 & 7,600,595  & 300.08 & 12.316 & 24.364 & 1.4129  & 1.8449  &  0.7592  & 0.6250  \\
         & Wikipedia2009(In) & 1,864,433  & 4,507,315   & 12.846 & 4.8903 & 2.6268 & 1.3988   &  -0.6291 &  1.9658  & 0.9891  \\
         & WikipediaLinkFr(In) & 4,906,478   & 113,122,279  & 1864.4 & 48.608 & 38.356 &  1.0988 & -0.7123  &  9.8888  & 0.9152  \\
         & Hudong(In) &   1,984,484 & 14,869,483   & 199.28 & 16.467 & 12.101 & 1.1567 &  10.921 & 0.0013   & 0.9883  \\
         \hline
         Biological & Yeast-PPIN & 2,361  & 7,182 & 8.0800 & 6.0838 & 1.3281 & 10.535  &  -0.4527 &  175.29  & 0.9930 \\
         Networks & Diseasome &  3,926 &  7,823 & 9.1009 & 5.5334 & 1.6447 &   10.9688 & -0.9493 & 134.52   & 0.8090  \\
         &  Bio-Mouse-Gene & 45,101  & 14,506,199 & 856.67 & 643.27 & 1.3317 & 6.3e-08  &  -1.2e-02 &  2.1e+00  & 0.9898 \\
         & Bio-Dmela & 7,393  & 25569  & 10.782 & 6.9170 & 1.5587 & 14.979  &  -0.5053 & 498.27  &  0.9806  \\
         & Bio-WormNet-v3 & 16,347  & 762,822 & 138.17 & 93.328 & 1.4805 & 5.6496   & -0.9801  &  704.71  & 0.9938 \\
         \hline
         Product  & amazon0601(In) & 403,394  & 3,387,388 & 15.279 & 8.3989 & 1.8191 & 3.8261  & -.7137  &  19.522  & 0.6010   \\
         co-purchasing  & amazon0505(In) &  410,236 & 3,356,828  & 15.313 & 8.1826 & 1.8714 & 3.8367  &  -0.8006 & 19.984  & 0.6880 \\
         networks  & amazon0312(In) &  400,727 & 3,200,444 & 15.073 & 7.9865 & 1.8873 & 3.7631  & -0.8179  & 18.747   & 0.5890 \\
         \hline

         Temporal  & sx-mathoverflow(In) &  24,818 & 506,550  & 31.476 & 10.424 & 3.0195 &  1.4452 & 2.4236  &  0.8241    & 0.9846  \\
         Networks  & sx-stackoverflow(In) & 2,601,977  & 63,497,050  & 186.00 & 27.647 & 6.7278 & 1.0218 & -0.8224  & 4.4865   & 0.9490 \\
           & sx-superuser(In) & 194,085 & 1,443,339 & 23.782 & 5.8239 & 4.0836 &  1.7401 &  2.1405 & 0.7284  & 0.9780 \\
           & sx-askubuntu (In) & 159,316 & 964,437 & 18.404 & 4.3856 & 4.1966  & 2.1923  &  2.2069 &  0.7665   & 0.9300 \\
        \hline
          Communication & Email-Enron & 36,692  & 183,831 & 36.100 & 10.021  & 3.6027  & 1.2417 &  -0.1275 & 2.9045  & 0.9641 \\
         Networks & Wiki-Talk(In) & 2,394,385  & 5,021,410 & 12.259 & 2.1195 & 5.7844 & 1.5167  &  -0.2846 &  0.0016  & 0.9900 \\
         & Rec-Libimseti(In) & 220,970  & 17,359,346 & 413.71  & 102.85 & 4.0227 &  2.5008 & -0.8496  & 331.18   & 0.9670  \\
         \hline
        Ground-truth  & Wiki-Topcats & 1,791,489  &  28,511,807 & 283.78  & 15.915 & 17.831 & 1.1811  &  0.1998 & 2.6412  & 0.8310 \\

        Networks  & com-Friendster & 65,608,366  & 1,806,067,135  & 137.81 & 55.056 &  2.5031 & 4.5863 & -0.9188  & 590.01  & 0.9000 \\

         & com-LiveJournal & 3,997,962  &  34,681,189 & 42.957 & 17.349 & 2.4759 & 2.8206 & -0.6020  & 65.638  & 0.7980  \\

        & com-Orkut & 3,072,441 & 117,185,083  & 154.78 & 76.281 & 2.0291 & 3.7049 & 0.1292  & 167.93 &  0.9890  \\

        & com-Youtube & 1,134,890 & 2,987,624  & 50.754 & 5.2650 & 9.6398 & 1.6113 & 8.3355  & 0.0094 &  0.8410 \\
       \hline
        Brain & Human25890-session1 & 177,584  & 15,669,036  &  319.01  & 176.47  & 1.8078  & 1.6098 & -0.2076 & 168.75   & 0.8710    \\
         Networks & Human25890-session2 & 723,881  & 158,147,409  & 667.91   & 436.94  & 1.5286  & 14.423 & -0.3466 & 18886.3   & 0.9980   \\
         & Human25864-session2 &  692,957 &  133,727,516 &  554.48  &  385.96 &  1.4366 & 16.250 & -0.3379  &  19217.8  & 0.9660   \\
         & Human25913-session2 & 726,197  & 183,978,766  &  446.92  & 258.99  & 1.7256  & 7.1013 & -0.4681  & 5779.8   & 0.9290    \\
         & Human25886-session1 & 780,185  & 158,184,747  &  558.41  &  405.50 & 1.3771  & 21.591 & -0.3119 &  26975.9  & 0.9768   \\
         \hline
    \end{tabular}
    }
    \label{table:estparameter}
\end{table}


Table \ref{table:estparameter} represents some of the statistical
measures corresponding to the network data and also provides the
statistical evidences of the proposed fitting over the node degree
distribution in the whole range using MLM distribution. CV is also
calculated corresponding to each of the degree distribution data and
it gives us the sufficient condition for the existence of the global
maximum at finite point of the $MLM (\alpha, \beta, \sigma)$
distribution. From Table \ref{table:estparameter} it is clear that
that the value of CV is greater than 1 in all the network data sets
under consideration. Thus it confirms that the maximum likelihood
estimates for the parameters $(\alpha, \beta, \sigma )$ of the
proposed MLM distribution attain at the finite points which has been
theoretically described in Section \ref{statistics_prop}\ref{existnce_mle_mlm}. To estimate the parameters
$(\alpha, \beta, \sigma )$ of the MLM distribution numerically, we
have used "optim" function along with the quasi-Newton L-BFGS-B
algorithm in R statistical software by taking the initial parameters
value $(\alpha, \beta, \sigma ) = (1,0,1)$. The estimated values of
the parameters for all the data sets satisfied the condition, i.e.,
$(\alpha > 0, \; \beta > -1 \; \text{and} \; \sigma > 0)$ as clearly
seen in Table \ref{table:estparameter}, for the complete
characterization of the proposed MLM distribution. Empirically it is
observed that in almost all the cases the estimated value of the
parameter $\sigma$ attains the higher values as compared to the
estimated value of $\alpha$. On the other hand, the estimated value
of the parameter $\beta$ lies between $(0,1)$ lies between $-1$ and
$1$ except a few which can be clearly seen from Table
\ref{table:estparameter}.

Furthermore, we leverage one of the popular statistical method viz.
bootstrapping chi-square test to evaluate the goodness-of-fit test
of the proposed MLM distribution. From Table
\ref{table:estparameter}, it is clear that the proposed MLM
distribution produces higher p values  (i.e. closure to 1) in almost
all the data sets which suggest that the null hypothesis i.e. the
data drawn from MLM distribution cannot be ruled out at the $0.05$
level of significance. This indicates that the observed degree
distribution is plausibly drawn from the MLM distribution. Thus from
Table \ref{table:estparameter} it can be concluded that the proposed
MLM distribution is effective in modeling the entire degree
distribution of real-worlld complex networks without ignoring some
of the lower degree nodes as oppose to the procedure of fitting
power law distribution. In addition, we also used some other
statical measures viz. KLD, RMSE and MAE in order to compare the
performance of the proposed MLM distribution with the each of the
other common power-law related distributions as given in the
following Tables \ref{table:compared1} and \ref{table:compared2}.


\begin{table}[h]
\tiny
\centering \caption{Table of different statistical measures of
different competitive models over real-worlld networks}
\makebox[\textwidth]{%
    \begin{tabular}{ |c|c|c|c|c|c|c|c|c|c|c|c|c|c| }
        \hline
         \multicolumn{2}{|c|}{Data} & \multicolumn{3}{|c|}{MLM}  & \multicolumn{3}{|c|}{Lomax} & \multicolumn{3}{|c|}{Power-law} & \multicolumn{3}{|c|}{Pareto} \\
         \multicolumn{2}{|c|}{sets} & \multicolumn{3}{|c|}{}  & \multicolumn{3}{|c|}{} & \multicolumn{3}{|c|}{} &  \multicolumn{3}{|c|}{}\\
         \multicolumn{2}{|c|}{} & \multicolumn{3}{|c|}{}  & \multicolumn{3}{|c|}{} & \multicolumn{3}{|c|}{} &  \multicolumn{3}{|c|}{} \\
         \multicolumn{2}{|c|}{} & \multicolumn{1}{c}{RMSE} & \multicolumn{1}{c}{KLD} & \multicolumn{1}{c|}{MAE} & \multicolumn{1}{c}{RMSE} & \multicolumn{1}{c}{KLD} &\multicolumn{1}{c|}{MAE} & \multicolumn{1}{c}{RMSE} & \multicolumn{1}{c}{KLD} & \multicolumn{1}{c|}{MAE} & \multicolumn{1}{c|}{RMSE}  & \multicolumn{1}{c}{KLD} & \multicolumn{1}{c|}{MAE}\\
         \hline

         Social & ego-Twitter(In) & 16.800  &  0.00819 & 1.3498 &  29.366  & 0.01354 & 2.4701 &  204.35  & 0.1831 & 10.847 & 354.25  & 0.2857  & 15.603    \\
         Networks & ego-Gplus(In) & 1.6115   &  0.05601 & 0.1825 &  10.491  & 0.06444 & 0.3033 &  53.064  & 0.2299 & 0.9221 & 86.955  & 0.3113  & 1.1847   \\
         &  soc-Slashdot  & 31.527   & 0.01365 &  2.3951 &  32.065  & 0.014102 &  2.4658 &  247.87 & 0.1007  & 10.074 & 247.84  & 0.1007  &  10.073  \\
         & soc-Delicious(In)  &  79.809 & 0.00839  & 3.7730 & 91.993   &  0.01326 & 4.8060  &  349.66 & 0.2021  & 14.867 & 471.02  & 0.1349 & 17.874   \\
          & soc-Digg(In)   & 13.634 & 0.02182  & 0.8440 & 24.841   &  0.02391 &  1.0269 & 208.01  &  0.1601 & 4.2185  &  212.87 & 0.1601 & 4.2312  \\
         & soc-Academia  & 16.323  & 0.00351  & 0.5705 & 48.951   & 0.01019 & 1.6178 &  229.54 &   0.2027 & 6.3889 & 440.15  & 0.274 &  10.464  \\
         & LiveJournal(In)  &  243.99  &  6.13e-04 & 5.4026  &  1764.9  & 0.02111  & 54.400  & 5025.2 & 0.1614  & 127.98  &  8100.9 & 0.1785 & 164.18   \\
         & Dogster-Friendship  & 32.203 & 0.01328 & 0.8502 &  36.449  & 0.01700 &  1.0755   & 358.27 & 0.2926  & 5.6815  & 549.57  & 0.4618  & 7.6734  \\
         & Higgs-Twitter(In) & 19.821  & 0.00785 & 0.4710 &  20.609  & 0.00793 & 0.4621  & 260.32 &  0.2492 & 4.8938  & 524.96  & 0.4806 & 7.6938   \\
         & Artist-Facebook & 11.708  & 0.01079 & 2.1381 & 12.923   & 0.01199 & 2.6173 & 100.49  & 0.1643 &  14.552 & 350.05  & 0.4010 & 26.467   \\
         & Athletes-Facebook & 4.4304  & 0.00879 & 1.3252 &  9.2379 & 0.00966 & 1.8260 &  100.16 & 0.2049 & 13.387 & 204.91  & 0.4164  & 23.839   \\
         \hline
         Citation & cit-HepTh(In) &  3.2640   & 0.01354 & 0.5071 &  7.9393  & 0.01585 & 0.7585  & 73.531  & 0.1741 & 4.0821 & 122.79 & 0.2566 & 5.997  \\
         Networks& cit-HepPh(In)& 9.6810  & 0.00821  & 1.9016 &  21.303  & 0.01317  & 3.1135  & 128.55   & 0.1825 & 13.234  &  257.41   & 0.2689 & 21.445  \\
         & cit-Patents(In)  & 445.80  & 1.61e-04  & 47.603 & 2577.5   &  0.00192 & 230.35 & 27.5K & 0.2266 & 2049.5 &   34.8K  & 0.2366 & 2533.2  \\
         & cit-Citeseer(In) & 40.728  & 0.00228  & 3.3778 &  28.032  & 0.00278 &  3.3902 & 889.88 & 0.3308 & 49.467 &  1156.2 & 0.2916 &  62.026  \\
         \hline
         Collaboration & ca-CondMat & 14.830  & 0.00479  & 4.1570 &  36.094 & 0.00814  & 7.3904   & 107.86 & 0.1025  & 26.092   & 469.12    & 0.3738 & 63.075  \\
         Networks & ca-AstroPh & 23.890  & 0.02756  & 5.9796 &  32.799  & 0.03457 & 7.4448  & 92.255   & 0.1753 & 15.158  &  251.03   & 0.3816 & 27.707   \\
         & ca-GrQc   & 15.850  & 0.03055  & 7.2247  &  35.935  & 0.04013 & 12.286  &  124.24  & 0.2554 &  27.137 & 202.33   & 0.2741 & 44.221  \\
         & ca-HepPh & 13.944  & 0.06959 & 4.1919 & 19.607 & 0.07266 & 4.7763 & 75.071   & 0.1769 & 8.0906  & 144.39  & 0.2569   & 14.668  \\
         & ca-HepTh   & 23.280  & 0.00896 & 10.851  &  61.797 &  0.01353 & 20.391  &  268.91  & 0.2346 & 66.108  &  437.19 & 0.2829  & 106.36   \\
         \hline
         Web & Google(In) & 360.62  & 0.01368   & 13.845 &  337.68  & 0.01546 & 14.201  & 1809.1  & 0.124 & 45.023    & 1809.2    & 0.124  & 45.023  \\
         Graphs & BerkStan(In) &  71.819 & 0.03116   & 0.9478 & 105.20   & 0.0346 & 1.1962  &  615.03  &  0.1863 & 4.0722  &   615.01  & 0.1863  & 4.0721  \\
         & Wikipedia2009(In) & 86.510  &  0.00169  & 7.7498  &  103.94  & 0.00197  &  8.5289 &  4371.9 & 0.1352  &  164.58 &  4371.9  & 0.1352  & 164.58 \\
         & WikipediaLinkFr(In) & 124.98 &  0.01776  & 0.3174 &  146.14 & 0.03082 & 0.4465  & 248.09  &  0.1518 & 0.7857  & 397.39   &  0.1521 & 1.0815 \\
         & Hudong(In)  & 8.0517 & 0.00433  & 0.2508  & 25.163  & 0.00525 & 0.4600 & 587.21  &  0.0868 & 4.6828  & 587.22  &  0.0868 & 4.6828  \\
         \hline
         Biological & Yeast-PPIN &  4.3766 & 0.01487   & 2.6321 &  12.651  & 0.02389 & 5.4529 &  75.325   & 0.1999 & 19.013  & 77.455 & 0.1998 & 19.079  \\
         Networks & Diseasome & 8.8683  & 0.08000  & 2.7445  &  12.451  & 0.10202 & 3.4575  & 26.006   & 0.2248 & 5.3567  & 26.005    & 0.2248 & 5.3566  \\
         & Bio-Mouse-Gene &  7.6919 & 0.18941  & 2.2557  & 14.654  & 0.19473 & 2.3919 &  41.371  & 0.4566 & 3.9018  &  92.539   & 0.5373 &  4.7342  \\
         & Bio-Dmela & 16.173  & 0.01305  & 4.0968  &  10.579  & 0.01759 & 3.8219  &  143.71 & 0.1907 &  21.415 &   143.67  & 0.1907  & 21.414  \\
         & Bio-WormNet-v3 &  14.054 & 0.04648  & 2.6066 &  13.018  & 0.09249 & 3.7468 & 46.259  & 0.2761  & 6.8867  &   101.89  &  0.3744 & 9.0163  \\
         \hline
         Product & amazon0601(In) & 94.347  &  0.00374 & 8.1863 & 147.602   & 0.00695  & 10.928  & 1495.4 & 0.2708 & 70.281   & 2539.8    & 0.4022 & 114.59 \\
         co-purchasing & amazon0505(In)  &  109.95 &  0.00412 &  9.1836 &  94.048  & 0.00499 & 8.7882 &  1572.9  & 0.2463 & 73.003  &  2494.5   & 0.3711 & 111.56 \\
         networks  & amazon0312(In) & 100.89 & 0.00430  & 8.5465  & 92.525   & 0.00495 & 8.5742 & 1564.4   &  0.2425 &  71.875 &   2462.9 &  0.3686 & 109.21  \\
         \hline
         Temporal  & sx-mathoverflow(In) &  19.706  & 0.01879  & 2.4647 & 38.764   & 0.02621 & 3.7877  & 213.91 & 0.2131 &  13.600  &  213.82   &  0.2132 &  13.612  \\
         Networks  & sx-stackoverflow(In) &  39.654  & 0.00336  & 0.8694  &  62.254  &  0.00345 & 1.0741  & 1877.5 & 0.2016 &  14.007  &  1884.7   & 0.2017 & 14.017  \\
         & sx-superuser(In) &  79.777  &  0.00654 & 4.5409   &  136.85  & 0.01045  & 6.8313  & 900.04  & 0.1808 & 33.837   & 900.33   &  0.1808 &  33.839  \\
          & sx-askubuntu(In) &  106.04 & 0.01100 & 6.2022  & 176.58  & 0.01707 & 9.3509  & 949.66  & 0.2091  &  39.419 &  949.73 & 0.2091 & 39.420 \\
         \hline
          Communication & Email-Enron & 74.667  & 0.03523   & 5.2075 &  76.155 & 0.03531  &  5.2347 & 246.51 & 0.1779 &  14.886  & 245.25    & 0.1778 &  14.859 \\
         Networks & Wiki-Talk(In)  &  670.47 &  0.00356 & 25.871  & 671.76  & 0.00357  & 25.898  & 9669.4   & 0.3376 &  293.63 & 9669.4   & 0.3376 &  293.63   \\
         & Rec-Libimseti(In) & 23.341  & 0.02163  &  0.4953   &  66.434 & 0.09978 & 1.7923  &   77.081  &   0.2198   &  2.1486   &  133.91   & 0.2096    &  2.7441  \\
         \hline
         Ground-truth & Wiki-Topcats & 11.375  &  0.00190  & 0.1011 & 14.955 & 0.00201 &  0.1347   & 565.21  & 0.1377 &  2.6145  &  930.44   & 0.1612  & 3.8073 \\
         Networks & com-Friendster &  8266.8  &   0.00126 &  411.15  &  41.69K & 0.06401  &  3385.2 &  71.5K  &  0.1498  &  4575.6   & 129K  & 0.1498 & 5591.7  \\
         & com-LiveJournal & 165.79 & 0.00084   &  6.9832 & 1741.3  & 0.02462  & 50.318 & 4102.9  & 0.1823 & 106.85  &  7116.4 & 0.2147  & 150.46  \\
         & com-Orkut &   197.89   & 0.00793   &  7.0113 & 207.43  &  0.01049 & 9.9761 & 2443.6   & 0.5498 &  80.712 & 4299.3   & 0.8033 & 101.64  \\
         & com-Youtube  & 53.288  &  0.00122  &  0.6984 &  81.409 & 0.00175  & 1.0862  &  1380.5  &  0.1342 &  15.690 &  1380.5  & 0.1342 & 15.691  \\
         \hline
         Brain & Human25890-session1 & 17.309 &  0.01920 & 3.7439 &  19.545 &  0.02264 & 4.0023 & 305.41   & 0.3397  & 22.537  &  588.71  & 0.5598   & 30.333   \\
         Networks & Human25890-session2 &  46.276  & 0.01024 & 5.8781  &  97.122 &  0.04774 & 15.303  &  794.95  & 0.4462 &  50.379 &  1623.8   & 0.6513  & 63.754   \\
         & Human25864-session2 & 64.037 & 0.01321  & 9.8670   & 111.45  & 0.05335  & 20.876  &  1120.3  & 0.4967 & 68.172  &  1711.3  & 0.6419  &  78.736  \\
         & Human25913-session2 &  112.661 &  0.01347  &  11.574  & 119.54  & 0.04719  & 20.971  & 904.76   & 0.2764  & 58.999  & 1892.4  & 0.4566 &  81.206  \\
         & Human25886-session1 & 65.181 & 0.01471  & 12.051 & 116.78  &  0.05476 & 23.396  &   978.66 &   0.4664 & 69.517  & 1873.6  & 0.6805 & 86.597  \\
         \hline
    \end{tabular}
    }\label{table:compared1}
\end{table}


Tables \ref{table:compared1} and \ref{table:compared2} depict the
values of different statistical measures (viz. RMSE, MAE and KLD)
which has been used for the measure of performances of the MLM
distribution in comparison to the competitive distributions while
modeling the data. RMSE and MAE are two different variants, carrying
information about the differences between actual and predicted
degree frequencies corresponding to a network. Higher similarity
between actual and mapped distributions is achieved by generating
smaller  values of RMSE and MAE. From Tables \ref{table:compared1}
and \ref{table:compared2}, it is clear that the proposed MLM
distribution provides smaller RMSE and MAE values compared to other
competitive distributions in almost all the networks except a few
where the power-law cutoff distribution outperforms the others. The
worst performance observed for the poisson distribution in
minimizing the RMSE and MAE values compared to the other competing
distributions over all the real-worlld networks as clearly seen from
Table \ref{table:compared2}.
The Kullback-Leibler divergence (KLD), or relative
entropy, is a quantity which measures the dissimilarity between two
probability distributions. Thus the smaller value of KLD represents
the higher similarity between the actual and the predicted
distribution. From Tables \ref{table:compared1} and
\ref{table:compared2} it is clear that the proposed MLM distribution
generates smaller KLD values compared to other competitive
distributions in almost all the networks except a few where
power-law cutoff distribution outperforms the others. This indicates
that the observed degree distribution satisfactorily matches the
proposed MLM distribution in almost all the networks. Note that, in
terms of KLD, the Poisson and Exponential distributions always
perform worse than the others in all the networks as in the case
RMSE and MAE.
The performance of the proposed MLM distribution is always superior to
the competitive in terms of KLD over almost all the networks. Thus
overall, by considering RMSE, MAE and KLD values, the performance of
the proposed MLM distribution for all the networks is found to be
better than the other competing distributions which suggest that the
observed distribution plausibly comes from the proposed MLM
distribution.

\begin{table}[h]
\tiny
\centering \caption{Table of different statistical measures of
different competitive models over real-worlld networks}
\makebox[\textwidth]{%
    \begin{tabular}{ |c|c|c|c|c|c|c|c|c|c|c|c|c|c| }
        \hline
         \multicolumn{2}{|c|}{Data} & \multicolumn{3}{|c|}{Log-normal}  &  \multicolumn{3}{|c|}{Poisson} & \multicolumn{3}{|c|}{Power-law} & \multicolumn{3}{|c|}{Exponential}\\
         \multicolumn{2}{|c|}{sets} & \multicolumn{3}{|c|}{}  & \multicolumn{3}{|c|}{} & \multicolumn{3}{|c|}{Cutoff} & \multicolumn{3}{|c|}{}\\
         \multicolumn{2}{|c|}{} & \multicolumn{3}{|c|}{}  & \multicolumn{3}{|c|}{} & \multicolumn{3}{|c|}{} & \multicolumn{3}{|c|}{} \\
         \multicolumn{2}{|c|}{} & \multicolumn{1}{c}{RMSE} & \multicolumn{1}{c}{KLD} & \multicolumn{1}{c|}{MAE} & \multicolumn{1}{c}{RMSE} & \multicolumn{1}{c}{KLD} & \multicolumn{1}{c|}{MAE} &\multicolumn{1}{c}{RMSE} & \multicolumn{1}{c}{KLD} & \multicolumn{1}{c|}{MAE} & \multicolumn{1}{c}{RMSE} & \multicolumn{1}{c}{KLD} & \multicolumn{1}{c|}{MAE} \\
         \hline

         Social & ego-Twitter(In) & 53.863  & 0.0169 & 2.9494  & 410.93   & 10.452  & 36.645 & 68.004  &  0.0397 & 4.1974 & 157.98 & 0.2733 &  11.567 \\
         Networks & ego-Gplus(In) &  10.155 & 0.0678 & 0.2523   &  95.967  & 25.317 & 3.0371 &  30.925 & 0.1475  &  0.6821 & 50.098 & 1.3131 &  1.8328 \\
         &  soc-Slashdot  & 237.63  &  0.1058 & 10.549 &  684.36 & 10.069 &  42.407 & 19.598  & 0.0075  &  1.3599  &  434.25 &  0.6381 & 22.275\\
         & soc-Delicious(In) & 281.34  & 0.0579   & 10.781  & 957.82  & 6.8432 & 56.634 &  66.896 & 0.0185 & 4.2366  &  535.11 & 0.4626  & 25.304 \\
         & soc-Digg(In)   & 69.438  & 0.0552   & 1.9087 & 323.59  & 21.541  & 15.134 & 65.713  & 0.0441 & 1.9042  &  204.50 &  0.8907  & 8.0015 \\
         & soc-Academia   & 91.003  & 0.0169   &  2.0921 &   542.38 & 7.2349 & 22.153  &  62.376 & 0.0255 &  1.9845 &  198.11 &  0.1924  & 6.6739 \\
         & LiveJournal(In)  & 3473.6   & 0.0355    & 70.64   & 13.61K & 9.1120 & 481.38 & 808.79  & 0.0101  & 24.501  & 7017.9   & 0.3449    & 186.01 \\
          & Dogster-Friendship  & 42.539   & 0.0272   & 1.1459  & 494.49 &  14.575 & 15.309  & 182.47  & 0.1862  & 4.1579  & 165.19  & 0.4765   & 5.4623  \\
          & Higgs-Twitter(In) & 41.955   & 0.0134  &  0.5995  & 448.91 &  16.309 & 14.753  &  118.23 &  0.0914 & 2.8051   &   134.68 & 0.3163  &  4.4689 \\
          & Artist-Facebook &  24.071 & 0.0154 & 3.026 &  323.46 & 12.537 & 53.920 & 56.801 & 0.0458 & 6.6452 & 88.351 & 0.1799  & 13.623  \\
          & Athletes-Facebook  & 15.461  & 0.0127  & 2.5674 &  175.95 & 4.7428 & 35.815 &  25.099 &  0.0324 & 4.3180  &  28.388 & 0.0586   & 6.2610  \\
         \hline
         Citation & cit-HepTh(In) &  22.59  & 0.0255 & 2.331 & 153.39   & 8.0679 & 13.774 &  25.42  & 0.0464 & 2.816 & 58.74  & 0.2778 & 4.5286  \\
         Networks& cit-HepPh(In) & 44.951 & 0.0189 & 4.4405  &  303.32  & 7.9234 & 46.775  & 36.887    &  0.0221 &  4.5287 & 107.32 & 0.1801 & 14.145 \\
         & cit-Patents(In) & 9612.7 & 0.0192 & 725.71 & 38.2K   & 1.6549  & 3657.1  & 2424.5    & 0.0061 & 271.89  & 13.2K & 0.0659 & 1147.5 \\
         & cit-Citeseer(In) & 353.26 & 0.0299 & 21.921 & 1507.6 & 2.566 & 109.15  & 195.67 & 0.0131 & 13.877 & 629.02 & 0.1486 &  44.301 \\
         \hline
         Collaboration & ca-CondMat & 42.665 & 0.0082 & 6.5746  & 378.65  & 2.7263 & 80.781   &  62.929  & 0.0287  & 13.362 & 64.985 & 0.0472 & 16.873\\
         Networks & ca-AstroPh & 28.209 & 0.0312 & 6.8565  &  229.81  &  9.8703 &  55.51 &  50.604  & 0.0384 & 7.4579 & 68.185 & 0.1235 & 14.361\\
         & ca-GrQc   & 30.184 & 0.0515  & 12.148    & 193.94  & 2.3256 & 63.61  &  58.305  & 0.0659  & 18.259 & 69.169 & 0.1418 & 25.759 \\
         & ca-HepPh  & 29.993 & 0.1011  & 6.8958  & 185.48  & 11.609  &  39.673 & 50.717 & 0.1128   & 8.2477 & 89.936  & 0.5187 & 17.589  \\
         & ca-HepTh  & 55.618 &  0.0178 &  21.032  &  370.15  & 1.5051 & 133.68  & 89.425  & 0.0245 & 27.613 & 109.96 & 0.0551  & 43.882  \\
         \hline
         Web & Google(In) & 1514.5 & 0.0878 & 40.067   & 4442.6 & 4.712 & 154.92   & 188.01    & 0.0157  & 9.6549 & 2589.4 & 0.4419 & 76.441 \\
         Graphs & BerkStan(In)  & 185.04 & 0.1002 & 2.0198   & 993.01   & 7.0379  &  11.9628 &  322.63   &  0.1037 & 2.8203 & 595.53 & 0.7438 & 6.6185 \\
         & Wikipedia2009(In) & 2720.9  & 0.0798 &  116.43   & 8425.7 & 3.6475  & 398.72 & 781.87  & 0.0082 & 35.531 &  4727.1 &  0.3431 & 213.66  \\
         & WikipediaLinkFr(In)  &  240.18 & 0.0543 &  0.5234  & 762.72 & 25.726  & 4.0278  & 121.31 & 0.0622 & 0.5006  & 534.61  & 1.0217 & 2.0471 \\
         & Hudong(In)  & 746.47  & 0.1593 & 6.5493  &  1975.73 & 11.088 & 25.837  & 75.362   & 0.0063 & 0.8323  & 1494.8  & 1.1798  & 15.836 \\
         \hline
         Biological & Yeast-PPIN & 29.928 & 0.0496 & 9.3869  & 109.62  & 2.5149 & 39.181  &  4.9595     &  0.0175 & 2.9178 & 45.462 & 0.1234 & 13.786 \\
         Networks & Diseasome & 23.282 & 0.1552 & 4.8906  &  55.985  & 3.0101 &  12.001 &  9.3332   & 0.0822  & 2.8587 & 31.709 & 0.2979 & 5.7013 \\
         & Bio-Mouse-Gene &  17.199 & 0.1878  & 2.5372  &  101.23  & 15.318  & 10.254  & 9.277    & 0.0943 & 1.6036  & 31.649  & 0.4882  & 3.6376  \\
         & Bio-Dmela & 46.271  & 0.0426  & 9.2857 &  206.44 & 3.6221  &  45.991 &  24.091  & 0.0162  &  5.0541 & 86.659  & 0.1724  &  18.048\\
         & Bio-WormNet-v3 & 17.726 & 0.0851  & 3.9352 &  104.46 & 18.563  & 21.927  &  6.7764 & 0.0419 &  2.2424 & 40.795  & 0.3082  & 7.0826 \\
         \hline
         Product & amazon0601(In) & 286.61 & 0.0102 & 16.881    & 2064.6 &  2.7267 & 140.46   &  297.39   & 0.0382 & 22.199 & 308.32 & 0.0574 & 24.114 \\
         co-purchasing & amazon0505(In)  & 358.59 & 0.0125 & 19.123 & 2172.5   & 3.0551  &  144.34 &   260.85  & 0.0342 &  20.136 &  390.13 & 0.0628 & 26.178 \\
         networks  & amazon0312(In) & 338.03 & 0.0116 & 17.742 &  2131.9  &  2.6839 & 140.75  &  273.39  &   0.0352 & 20.381 & 383.82  & 0.0639 &  26.299 \\
         \hline
         Temporal  & sx-mathoverflow(In) & 41.934 & 0.0634 & 5.3161  &  281.78 &  8.0773 &  32.868  &  92.603   &  0.0861 &  7.9912 &  129.69 & 0.4636 &  15.172 \\
         Networks  & sx-stackoverflow(In) & 341.96 & 0.0286 & 4.4267  & 2469.6 & 18.054 & 42.111   & 740.18    & 0.0685  & 7.2275  & 1324.4 & 0.6829 & 19.362  \\
         & sx-superuser(In) & 243.42 & 0.0616 & 13.199  & 1246.2 & 3.8103 & 68.010   &   354.72  & 0.0570 & 16.613 & 609.40 & 0.3891 & 34.655 \\
          & sx-askubuntu(In) & 212.91  & 0.0649 & 12.451  & 1228.7  & 2.6253 & 68.973  & 389.14  & 0.0719 & 20.113 & 555.44  & 0.3433 & 33.693 \\
         \hline
          Communication & Email-Enron & 121.47 & 0.0873 & 8.445   & 426.39  & 6.8601 & 38.373   &  95.468   & 0.0689 & 7.664 & 230.41 & 0.5405 & 18.139 \\
         Networks & Wiki-Talk(In)  & 7978.6  &  0.1902 & 246.26   &  21.9K  & 1.2506 & 646.54   &  672.32  & 0.0036 & 25.905 & 16.5K & 0.4879 & 542.31 \\
         & Rec-Libimseti(In) & 87.472  & 0.0755  & 1.4021  &  281.18  & 30.222   & 8.0019    & 28.059  & 0.0359  & 0.6971 & 166.18  &  0.8547 & 3.9402 \\
         \hline
          Ground-truth & Wiki-Topcats  & 272.99 & 0.0464 & 1.5159  & 1477.2 & 8.7468 & 12.121   &  389.86   & 0.0629  & 2.2289 & 832.23 & 0.6767 & 5.8936\\
         Networks & com-Friendster & 101K & 0.0762 & 4022.8 &  280K & 24.658  &  22.1K  & 17.8K  & 0.0052 & 1025.5 &  193K &  0.7216 &  10.1K \\
         & com-LiveJournal & 2629.9 & 0.0299 & 51.656 & 10.9K  &  9.5778 & 401.89  &  497.89 & 0.0104  & 18.559  & 5230.3 & 0.2889 &  139.74 \\
         & com-Orkut & 452.92 & 0.0459 & 19.624  &  3118.1  & 11.839 & 135.17  & 261.75  & 0.0479 & 16.197 & 228.83 & 0.0599 & 14.496 \\
         & com-Youtube  & 1422.2 & 0.1416 & 17.219  &   3838.9 & 3.4522  & 51.118  & 143.79   & 0.0045 & 2.2564 & 2515.4 & 0.6241 & 31.101 \\
         \hline
         Brain & Human25890-session1 &   27.703 &  0.0222 &  4.1822  & 472.32   & 16.412  & 52.714 & 65.972 &  0.0509 & 7.8146  & 92.649   & 0.2272 &  15.977   \\
         Networks & Human25890-session2 & 78.483  & 0.0471 & 14.289   &  1326.6  & 16.151 & 112.92 & 83.707  & 0.0162 & 6.709  & 183.06 & 0.1701  & 26.302 \\
         & Human25864-session2 & 83.489 & 0.0495 & 18.215   & 1433.3   & 17.298 & 143.91 &  106.20 & 0.0156  &  9.8828  &  212.96 & 0.1605 &  33.272 \\
         & Human25913-session2 & 99.615  &  0.0292 &  15.796  & 1614.4  &  18.899 &  171.16 & 223.39 & 0.0219 & 15.122 & 440.58  & 0.3331 & 59.629 \\
         & Human25886-session1 &  89.805 &  0.0568 &  20.819  & 1568.3  & 13.481  & 153.03   &  102.34 & 0.0154 & 11.296 & 207.52  &  0.1354 & 33.287  \\
         \hline
    \end{tabular}
    }\label{table:compared2}
\end{table}

\begin{figure}[H]
\centering \makebox[\textwidth]{
\includegraphics[width=7cm, height=5cm]{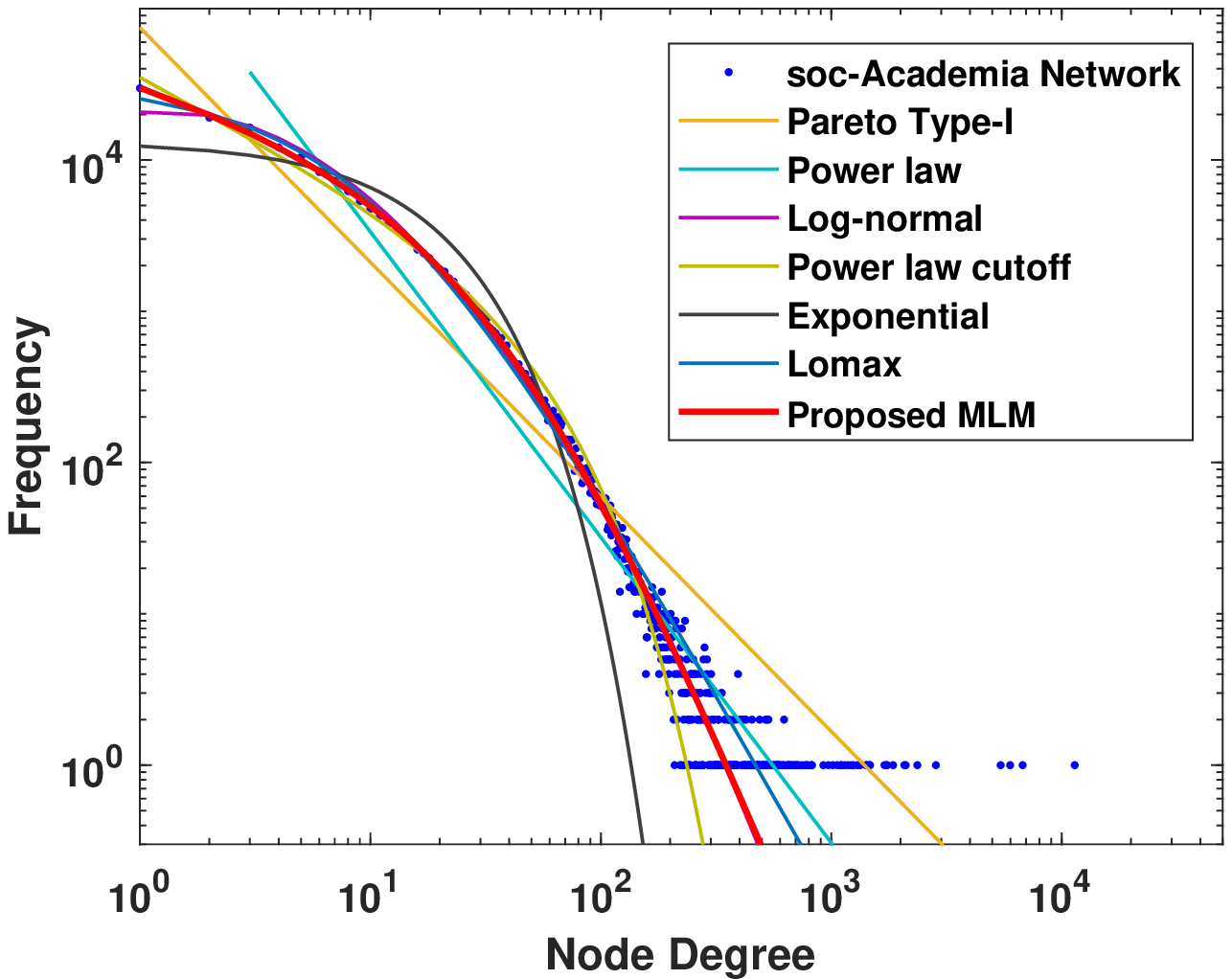}
\includegraphics[width=7cm, height=5cm]{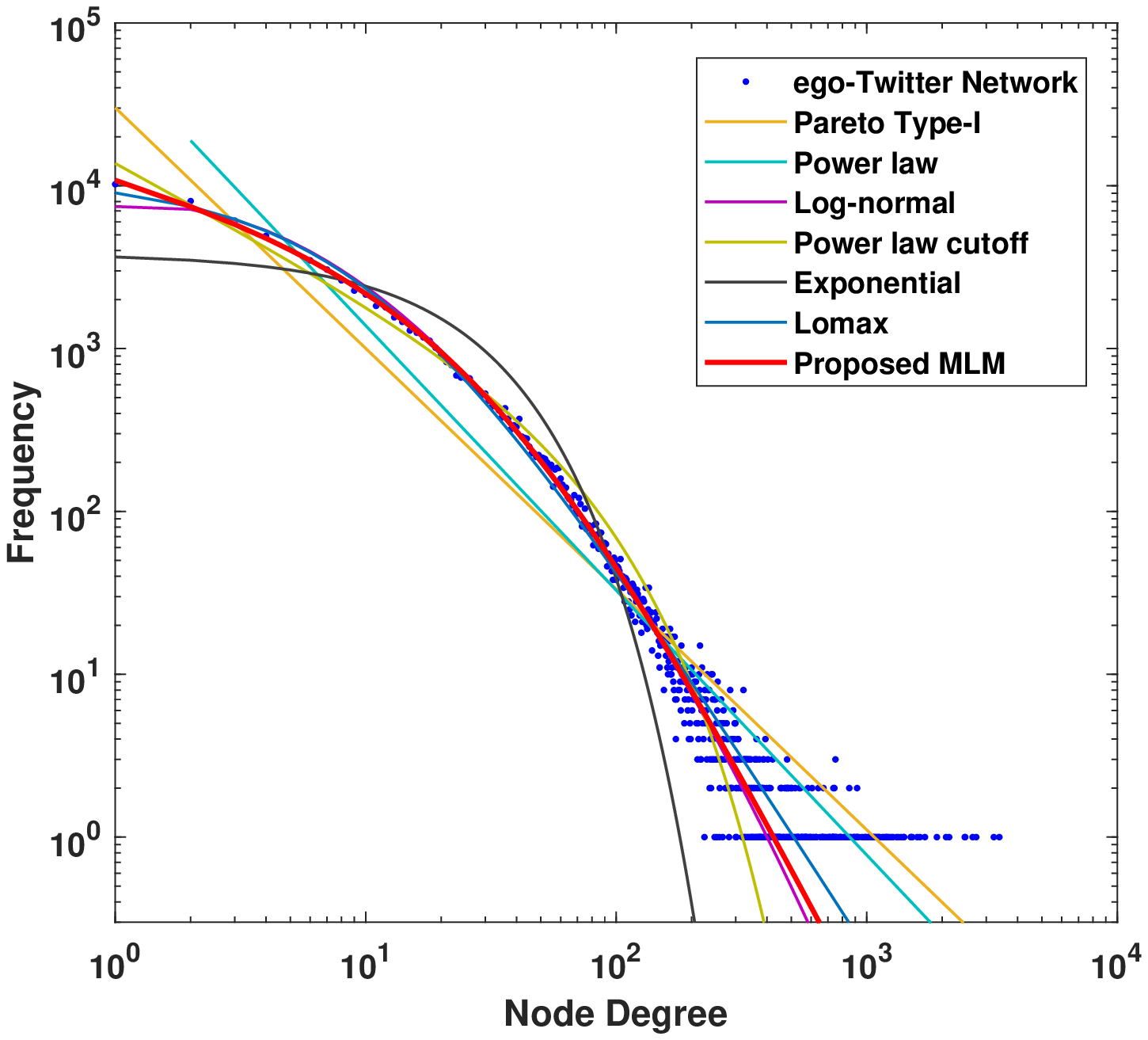}
} \caption{Degree distribution of soc-Academia and ego-Twitter
networks in log-log scale} \label{fig_mlm_1}
\end{figure}

\begin{figure}[H]
\centering \makebox[\textwidth]{
\includegraphics[width=7cm, height=5cm]{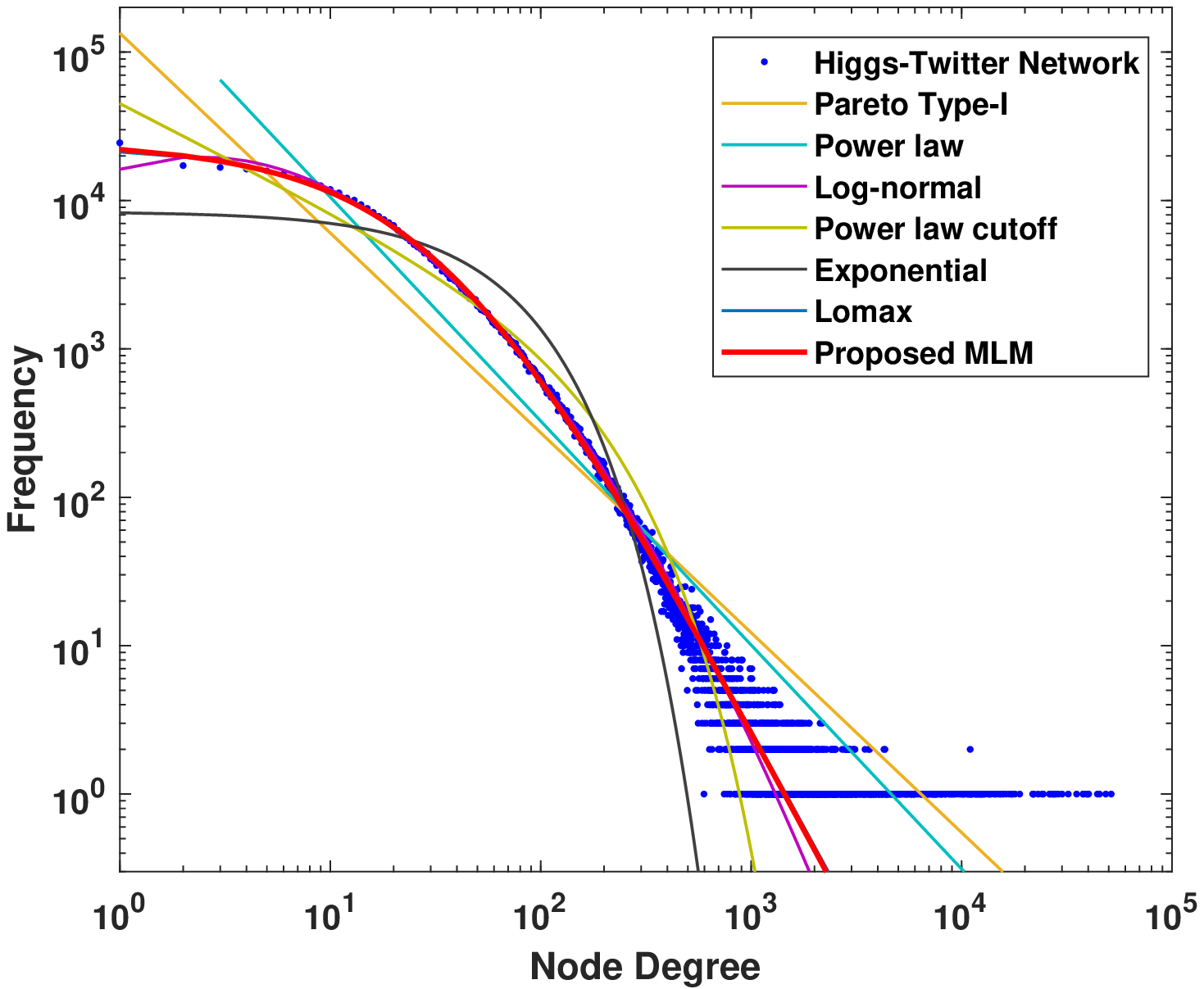}
\includegraphics[width=7cm, height=5cm]{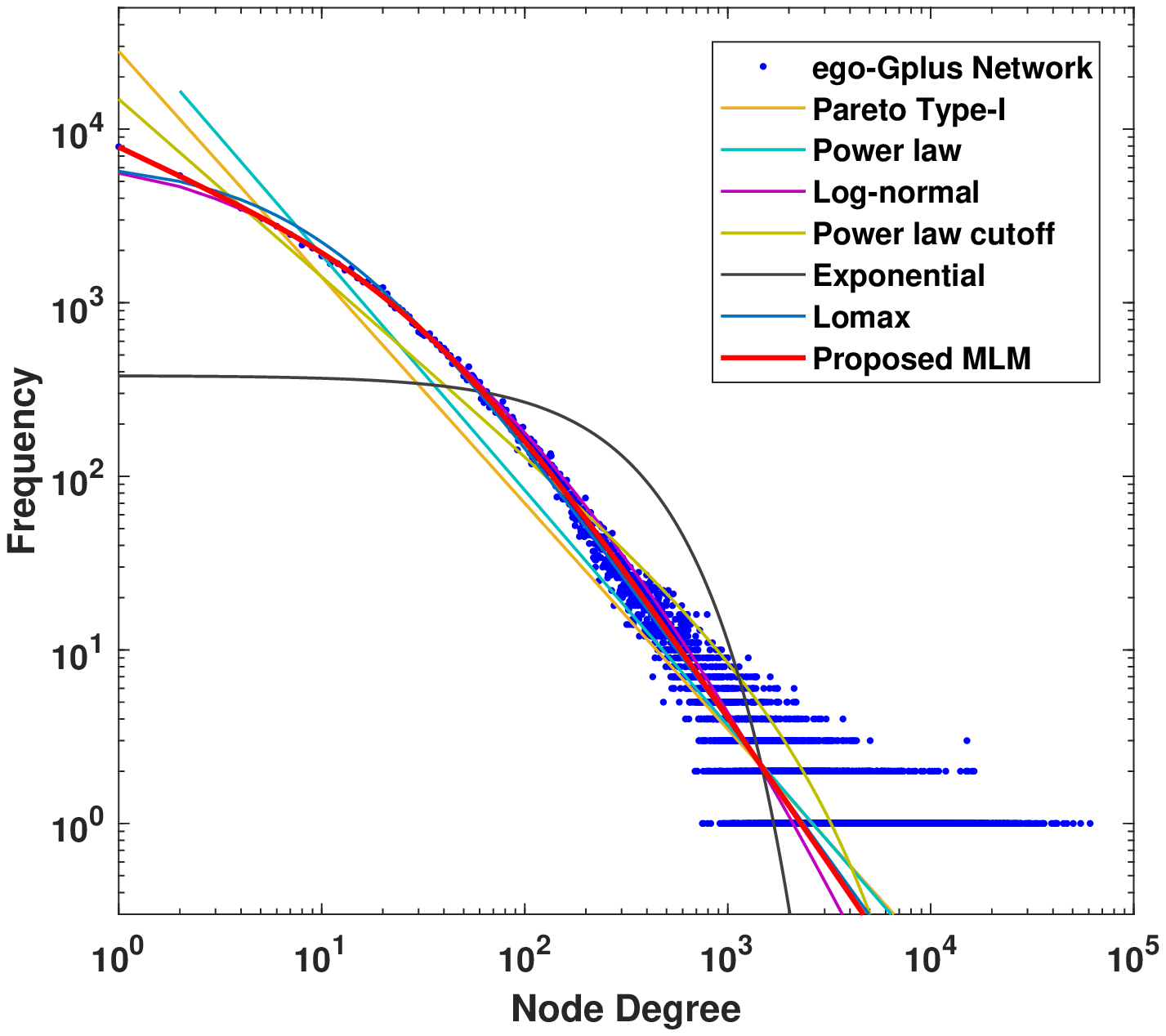}
} \caption{Degree distribution of Higgs-Twitter and ego-Gplus
networks in log-log scale} \label{fig_mlm_2}
\end{figure}

\begin{figure}[H]
\centering \makebox[\textwidth]{
\includegraphics[width=7cm, height=5cm]{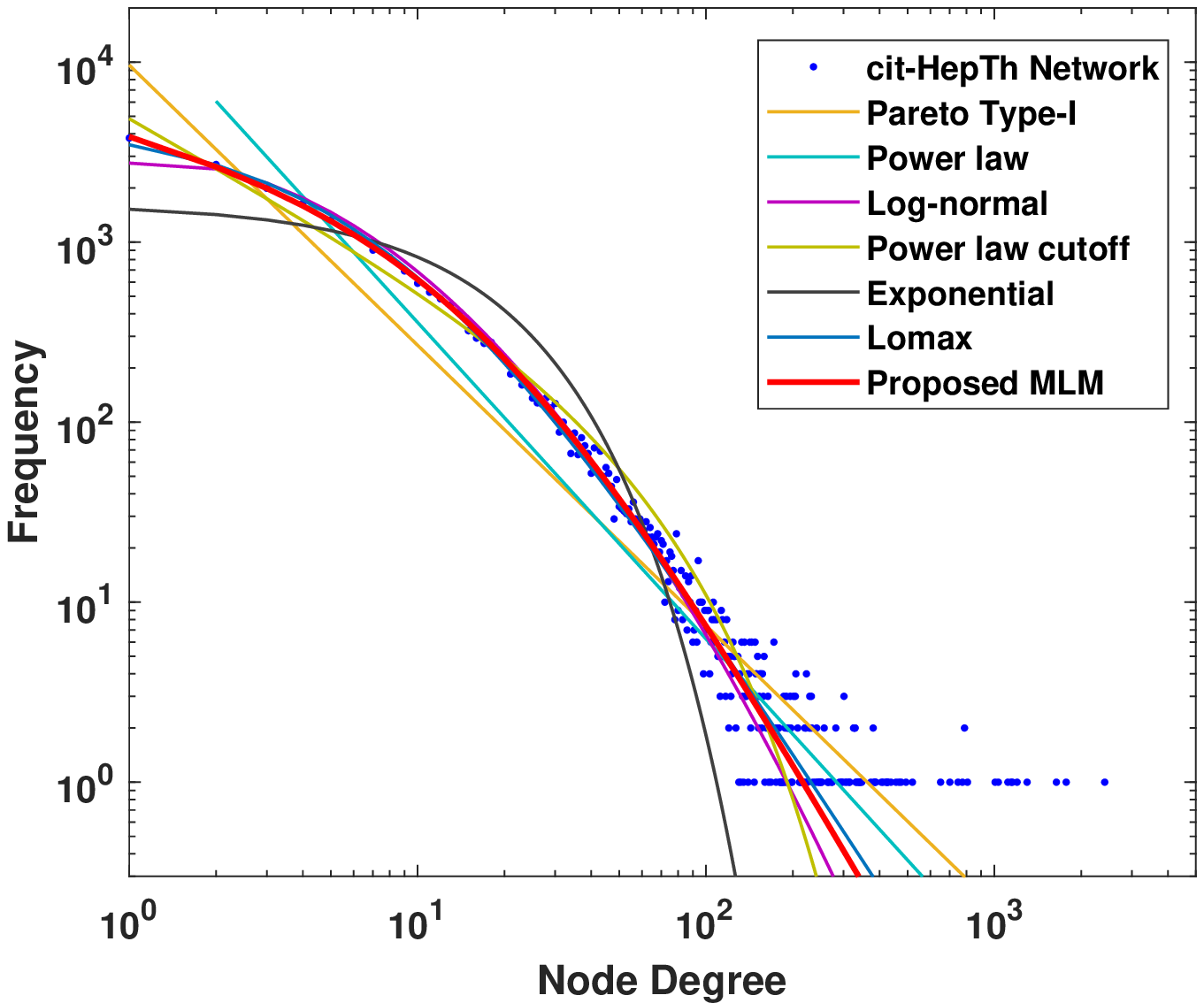}
\includegraphics[width=7cm, height=5cm]{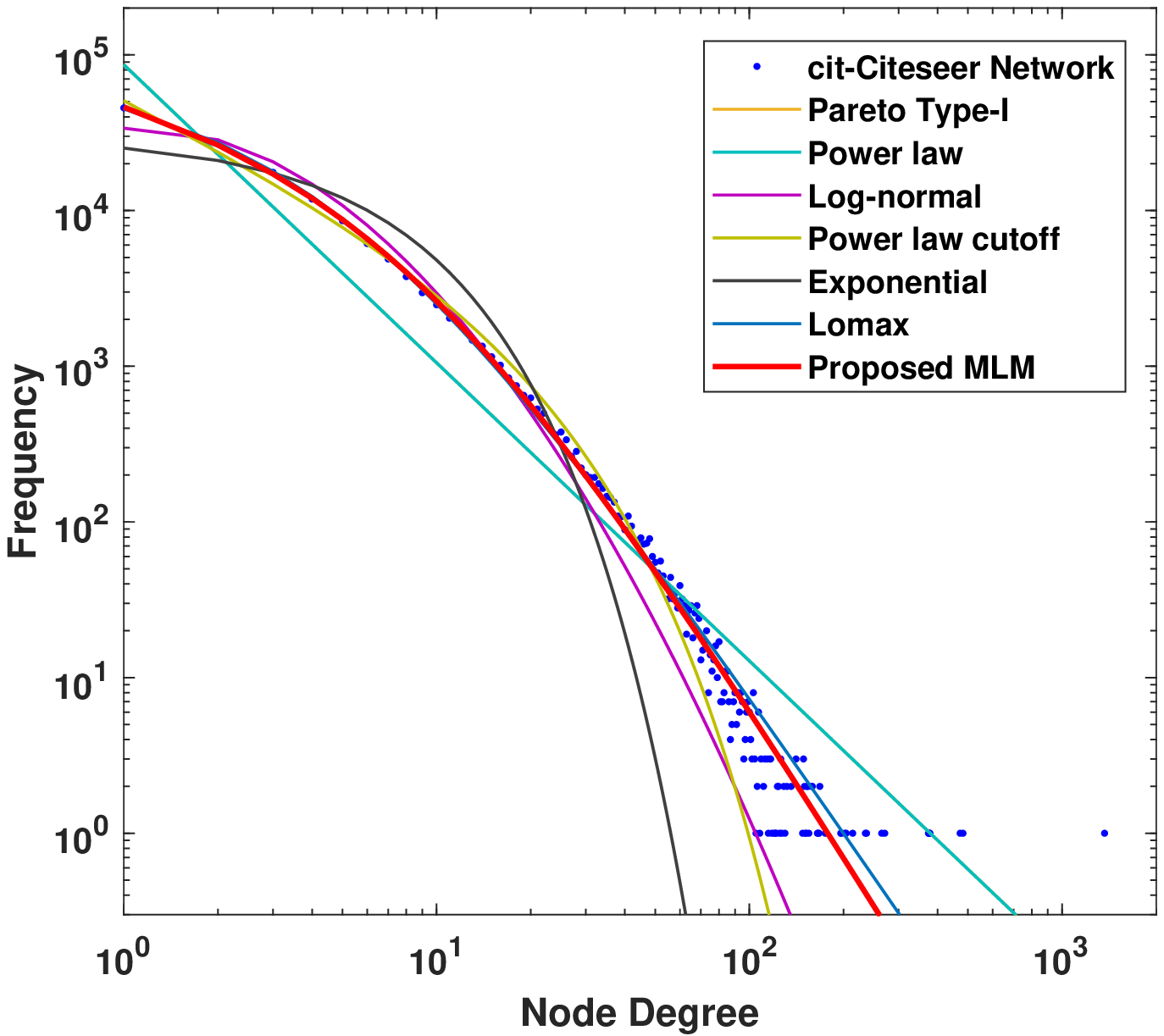}
} \caption{Degree distribution of cit-HepTh and cit-Citeseer
networks in log-log scale} \label{fig_mlm_3}
\end{figure}


\begin{figure}[H]
\centering \makebox[\textwidth]{
\includegraphics[width=7cm, height=5cm]{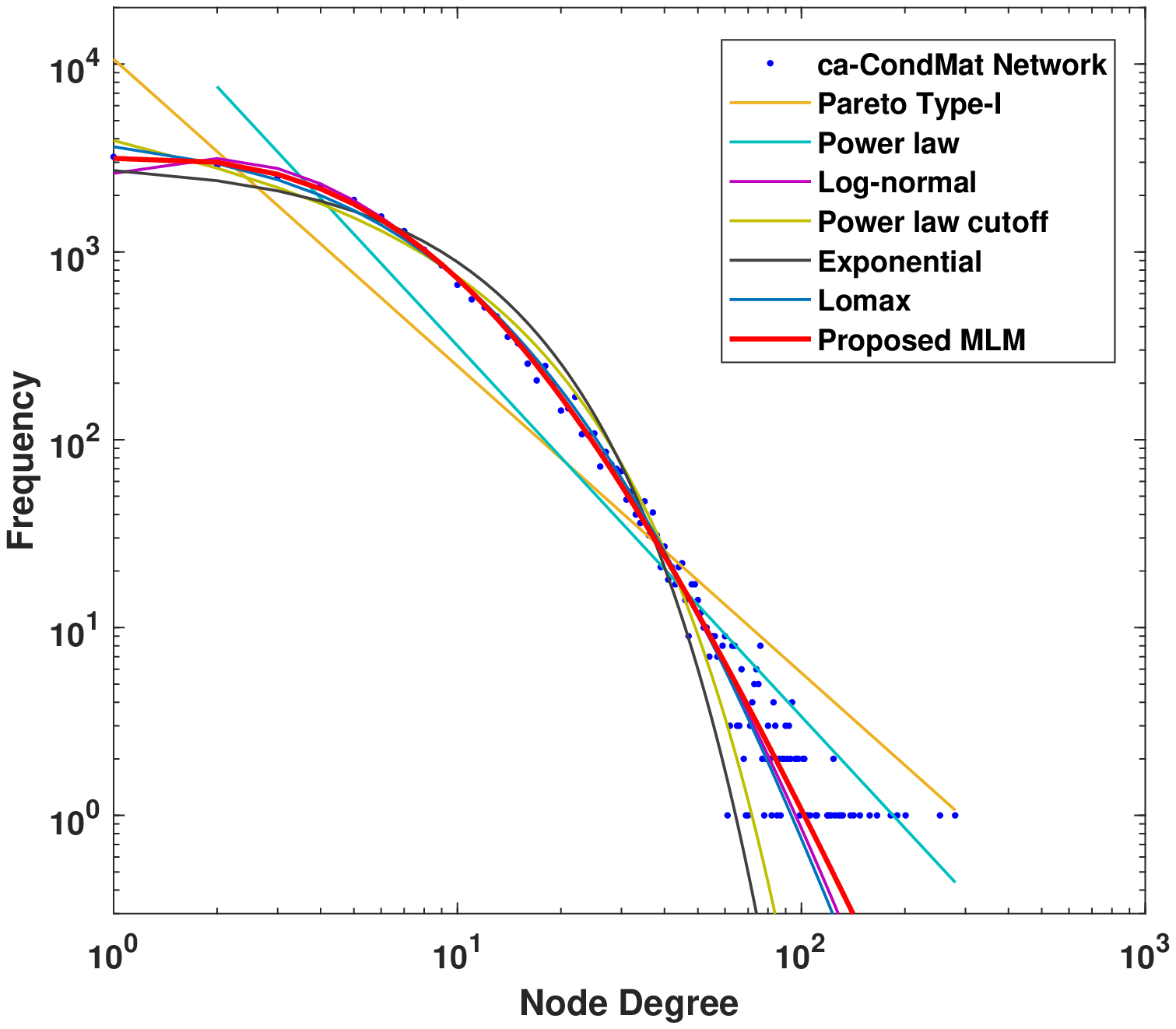}
\includegraphics[width=7cm, height=5cm]{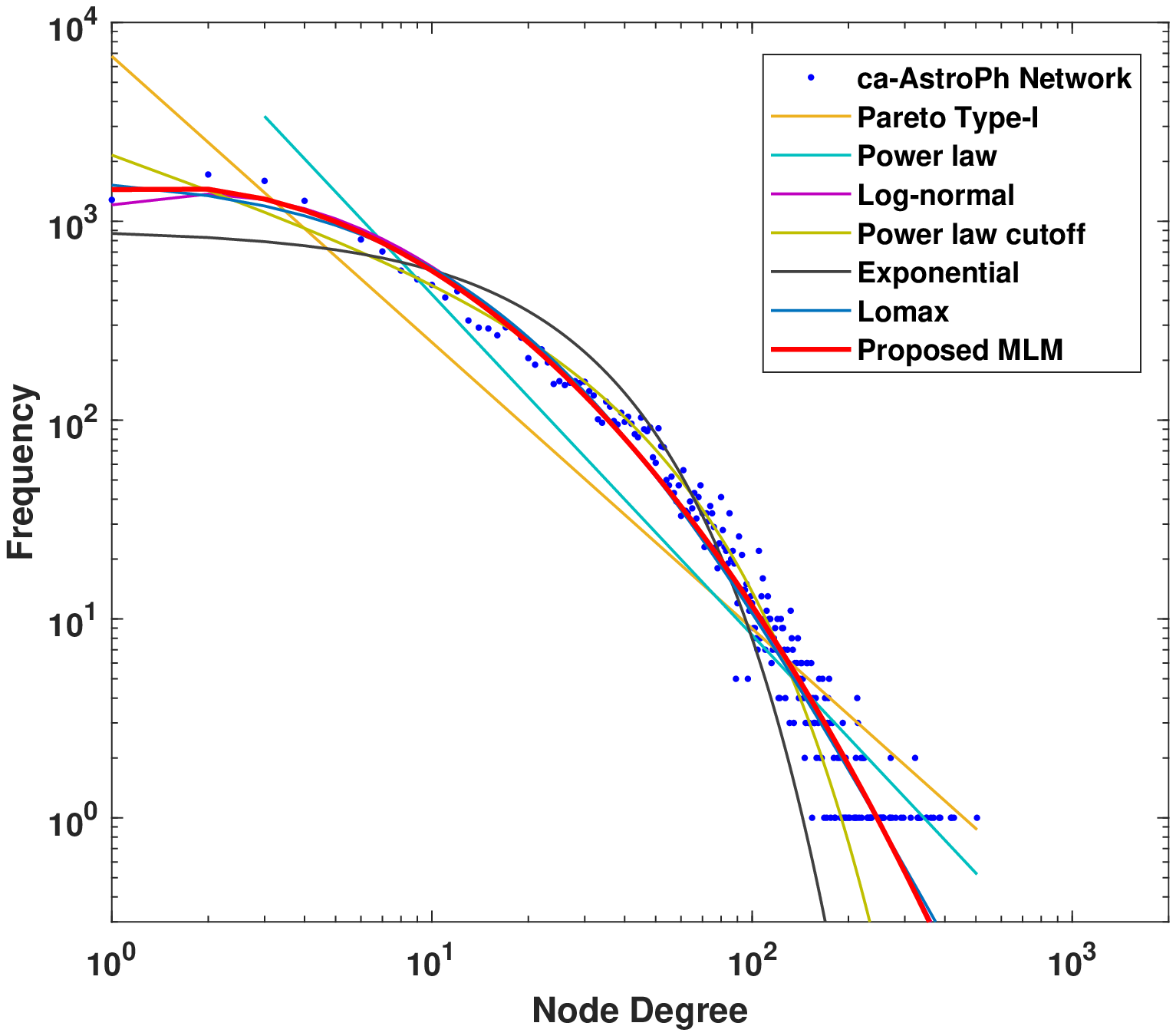}
} \caption{Degree distribution of ca-CondMat and ca-AstroPh networks
in log-log scale} \label{fig_mlm_4}
\end{figure}


\begin{figure}[H]
\centering \makebox[\textwidth]{
\includegraphics[width=7cm, height=5cm]{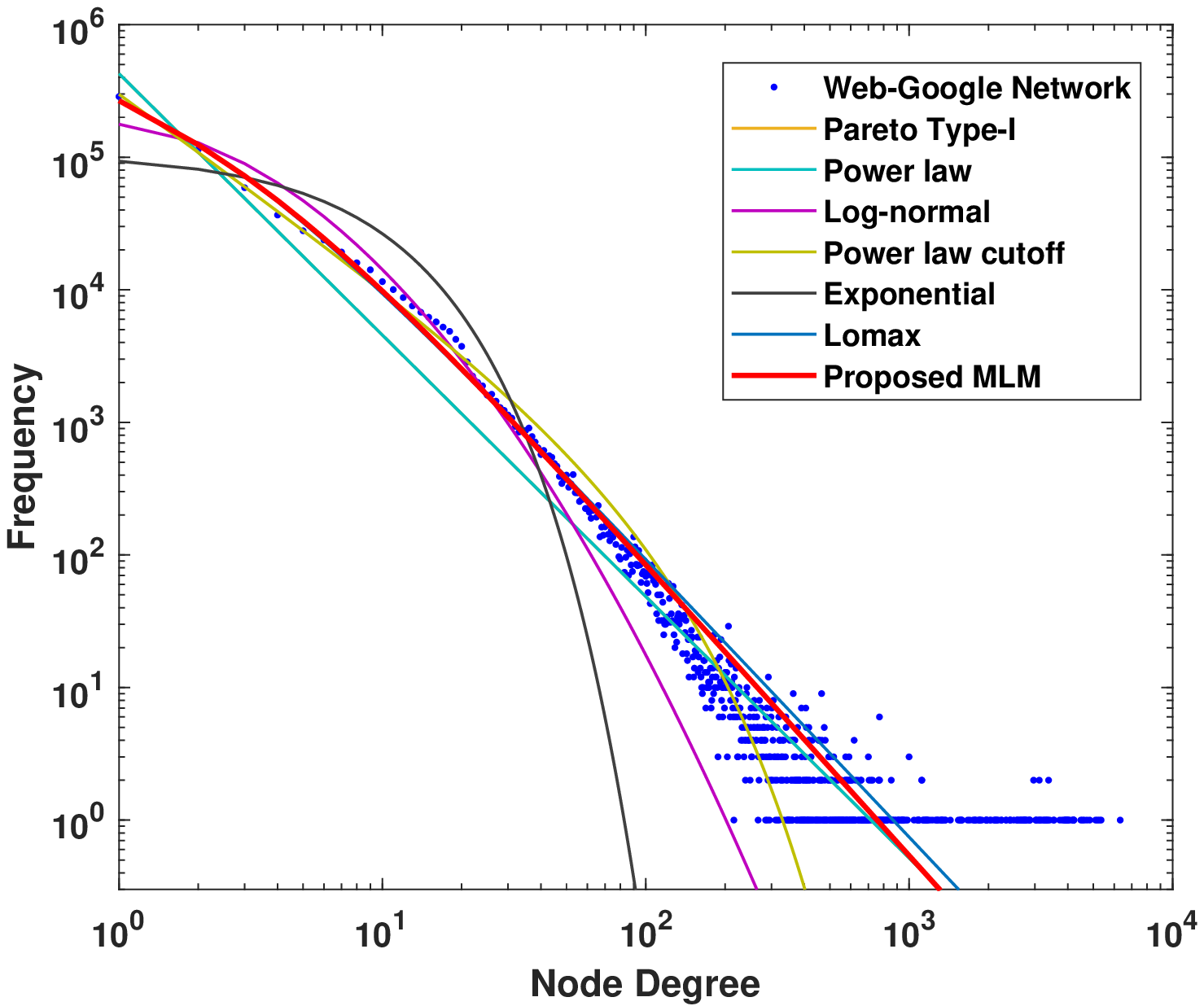}
\includegraphics[width=7cm, height=5cm]{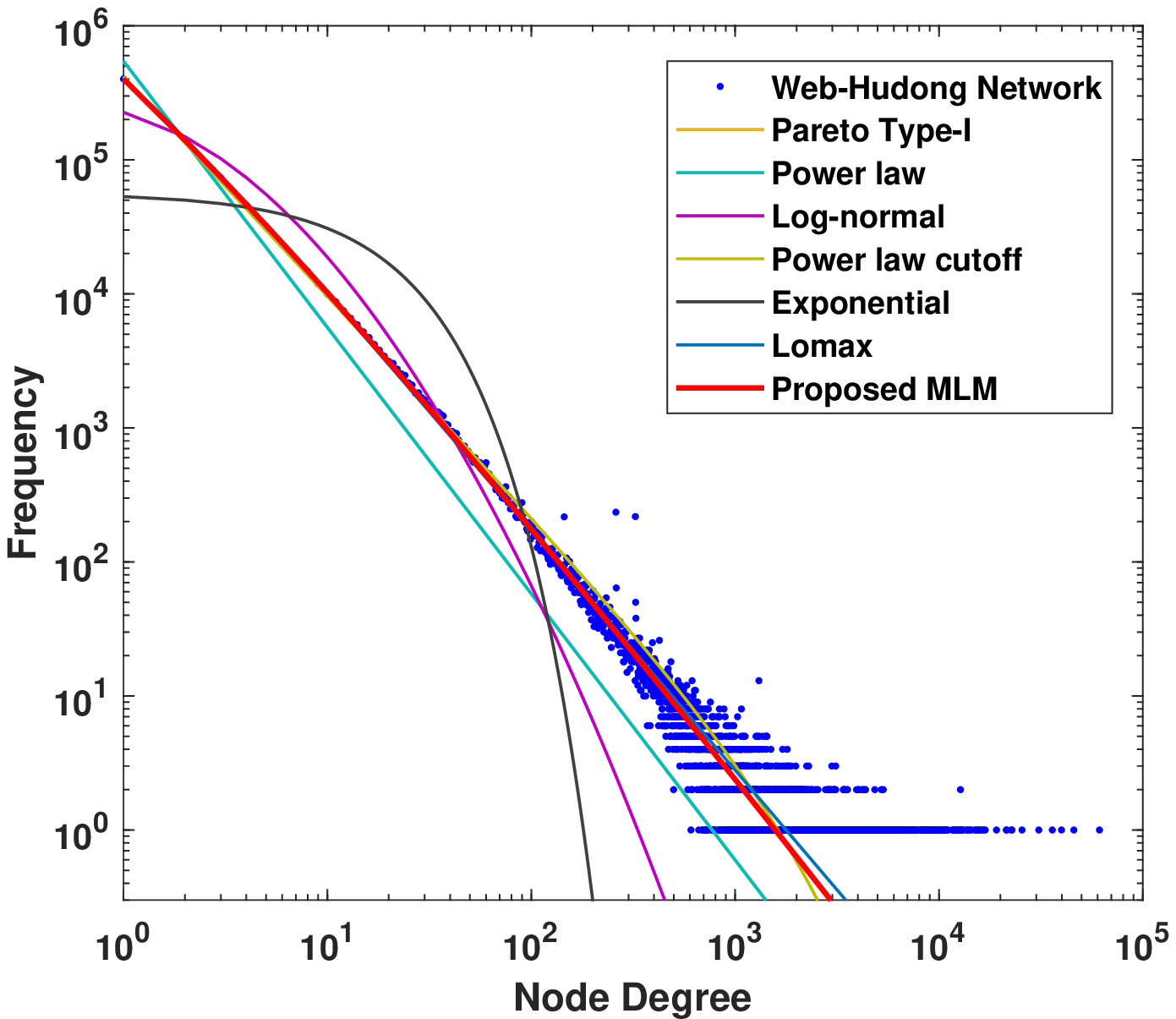}
} \caption{Degree distribution of Web-Google and Web-Hudong networks
in log-log scale} \label{fig_mlm_5}
\end{figure}

\begin{figure}[H]
\centering \makebox[\textwidth]{
\includegraphics[width=7cm, height=5cm]{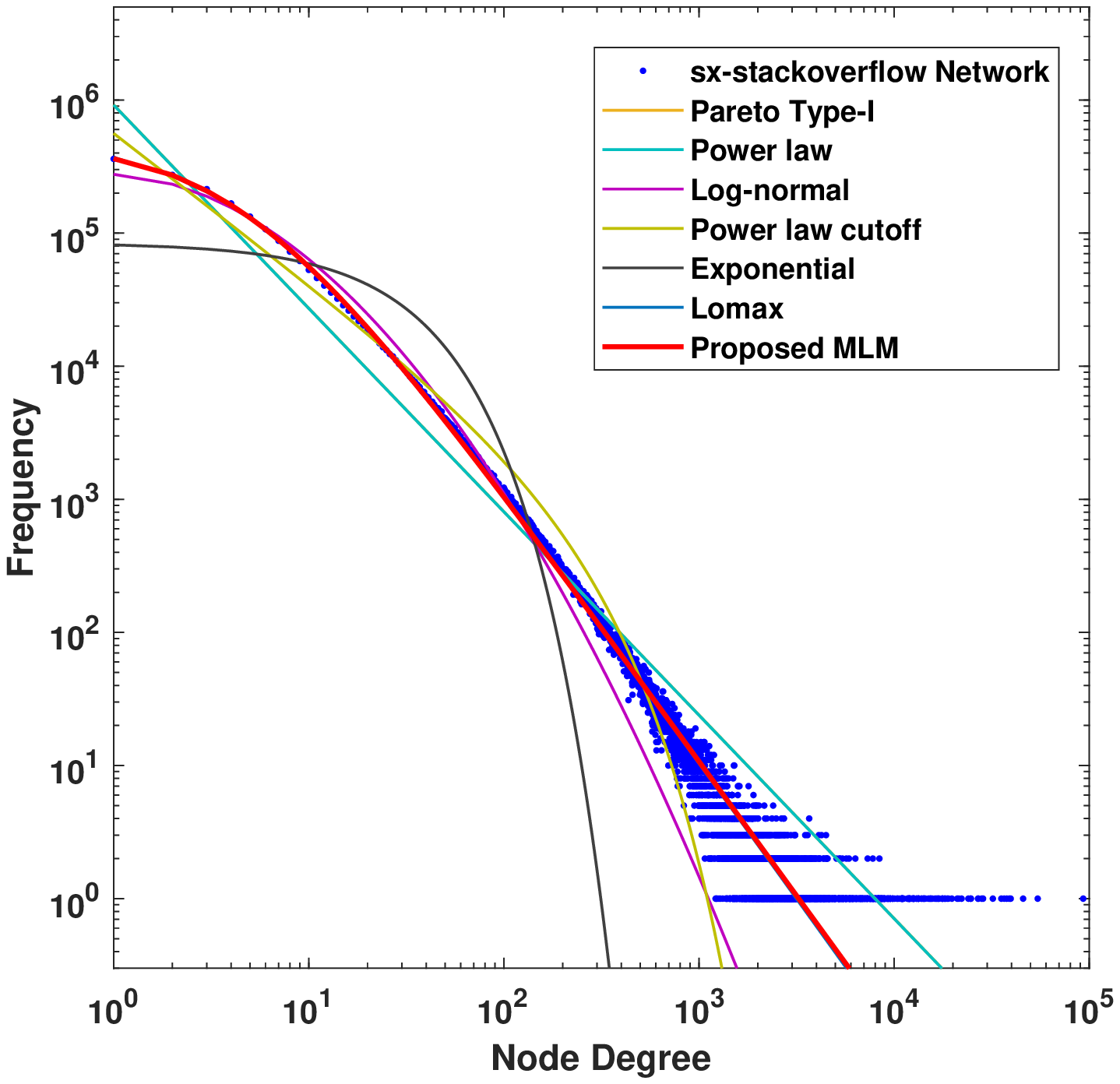}
\includegraphics[width=7cm, height=5cm]{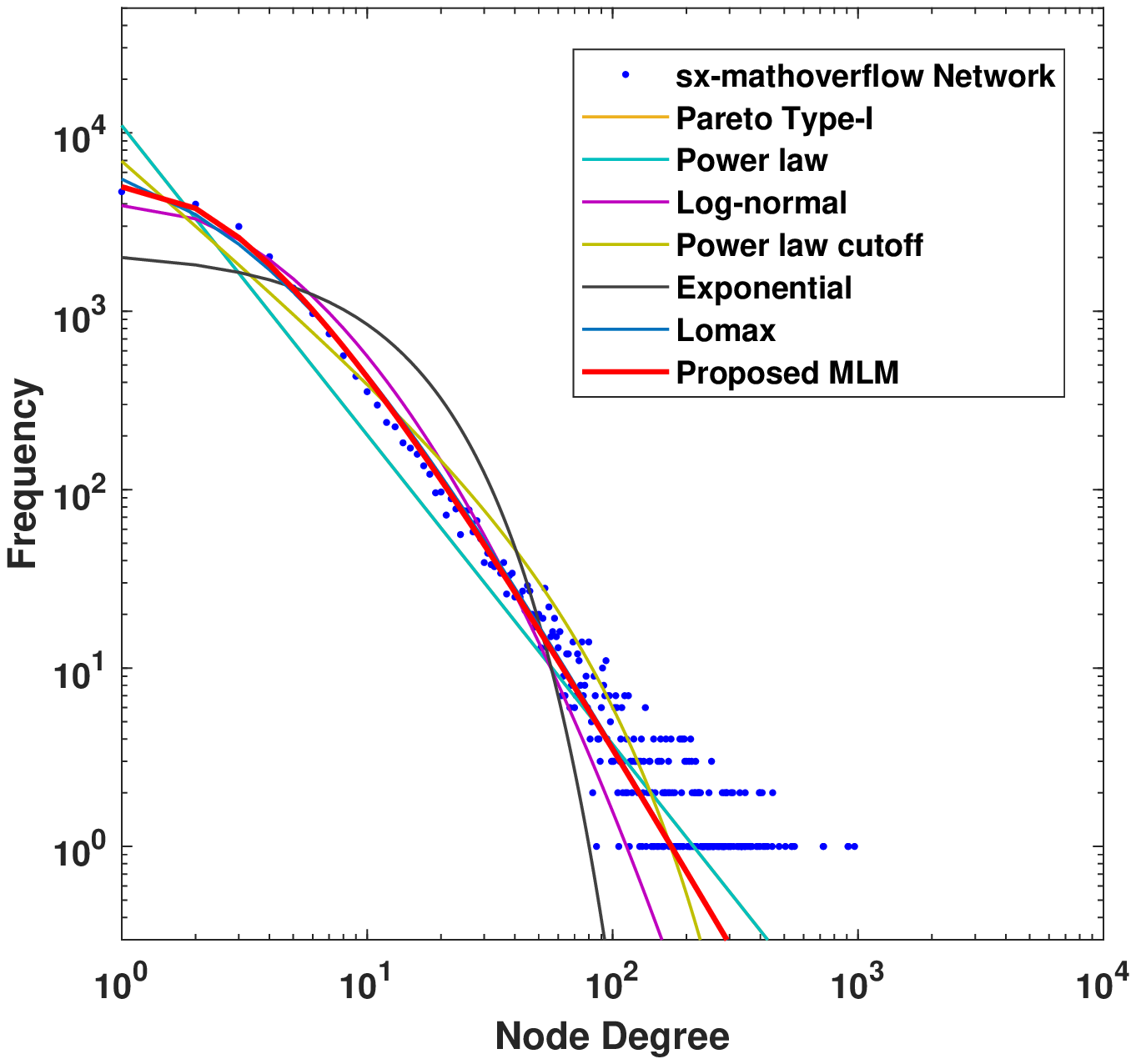}
} \caption{Degree distribution of sx-stack overflow and
sx-mathoverflow networks in log-log scale} \label{fig_mlm_6}
\end{figure}

\textbf{\begin{figure}[H] \centering \makebox[\textwidth]{
\includegraphics[width=7cm, height=5cm]{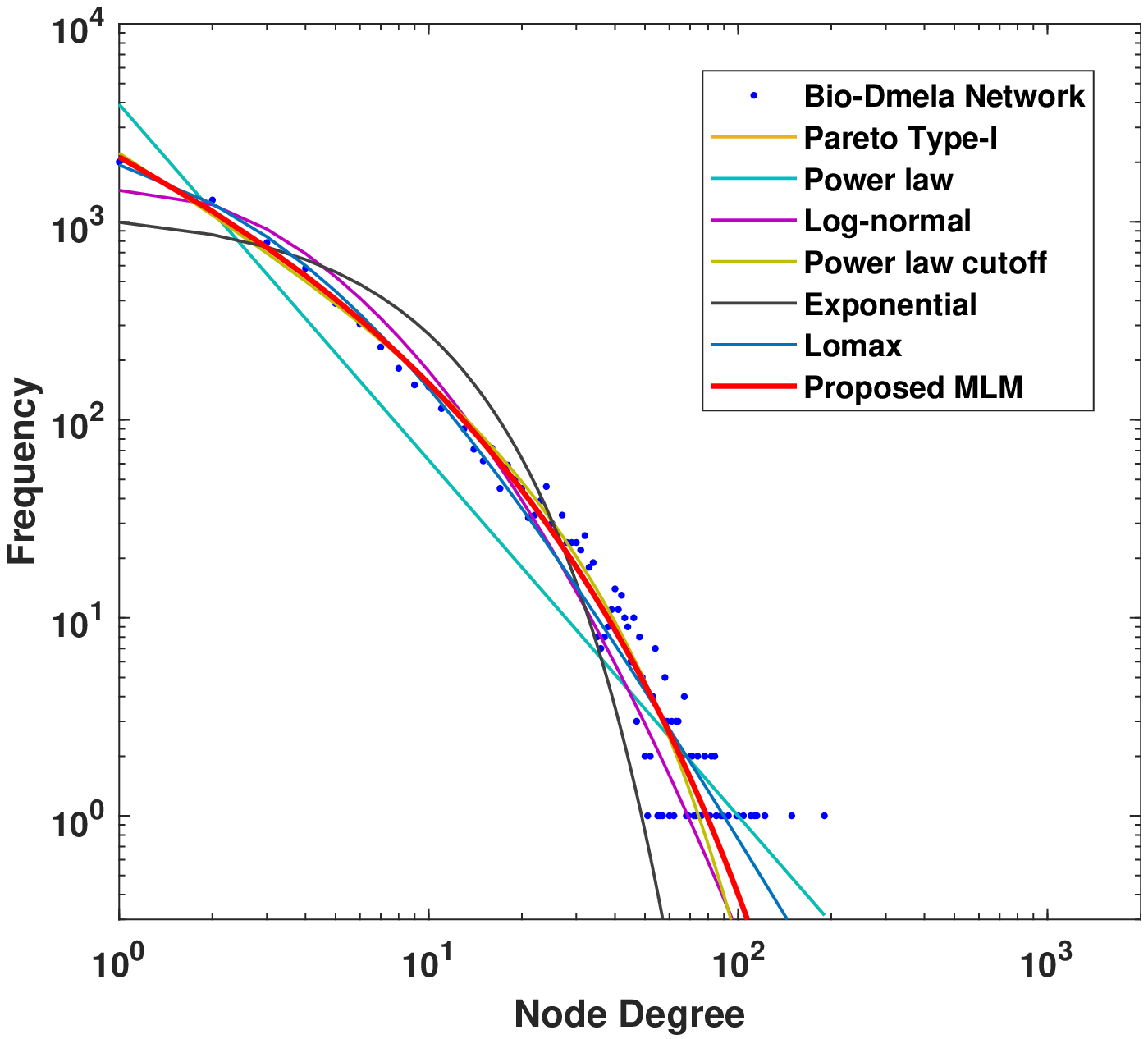}
\includegraphics[width=7cm, height=5cm]{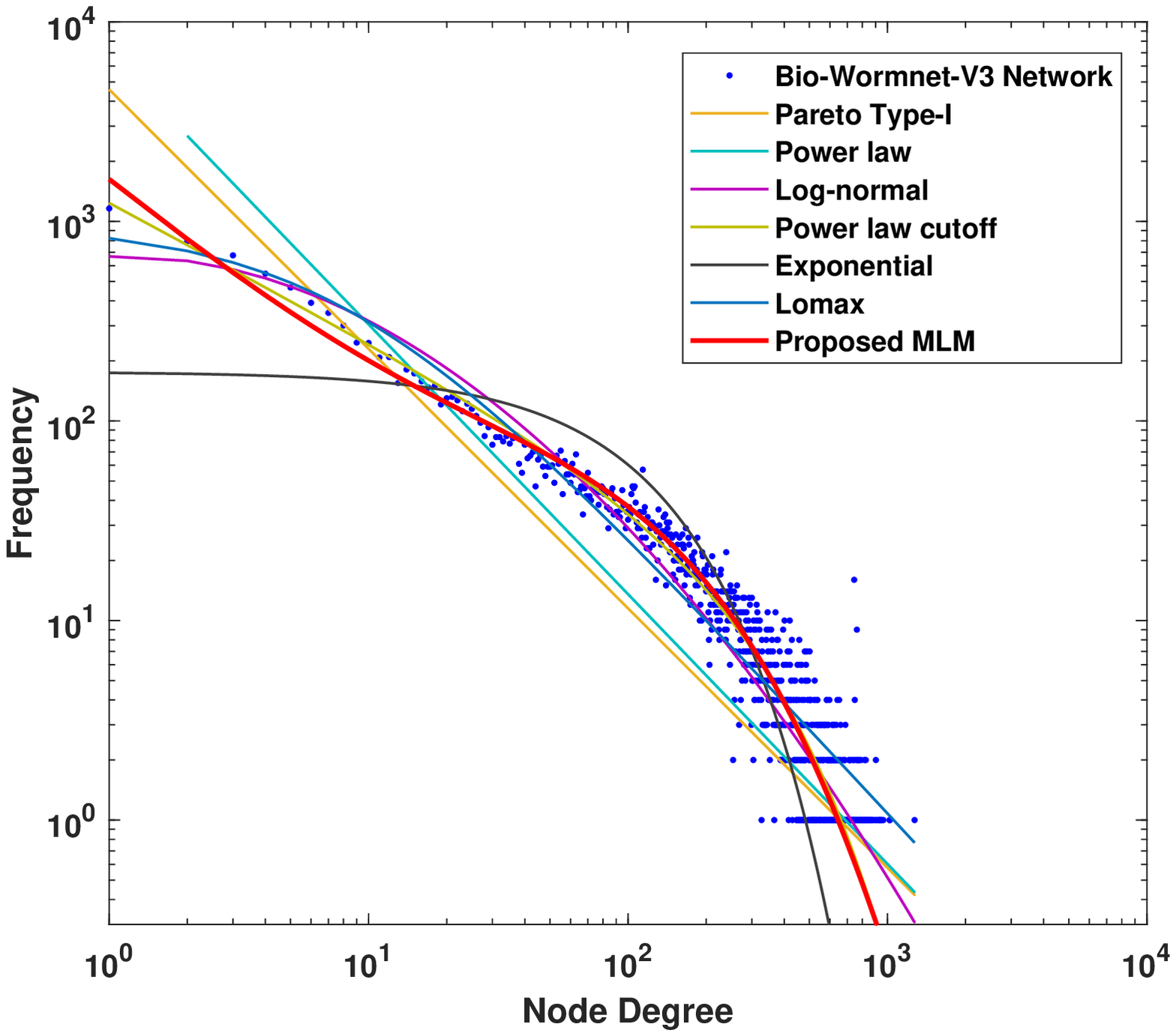}
} \caption{Degree distribution of Bio-Dmela and Bio-Wormnet-V3
networks in log-log scale} \label{fig_mlm_7}
\end{figure}}

\begin{figure}[H]
\centering \makebox[\textwidth]{
\includegraphics[width=7cm, height=5cm]{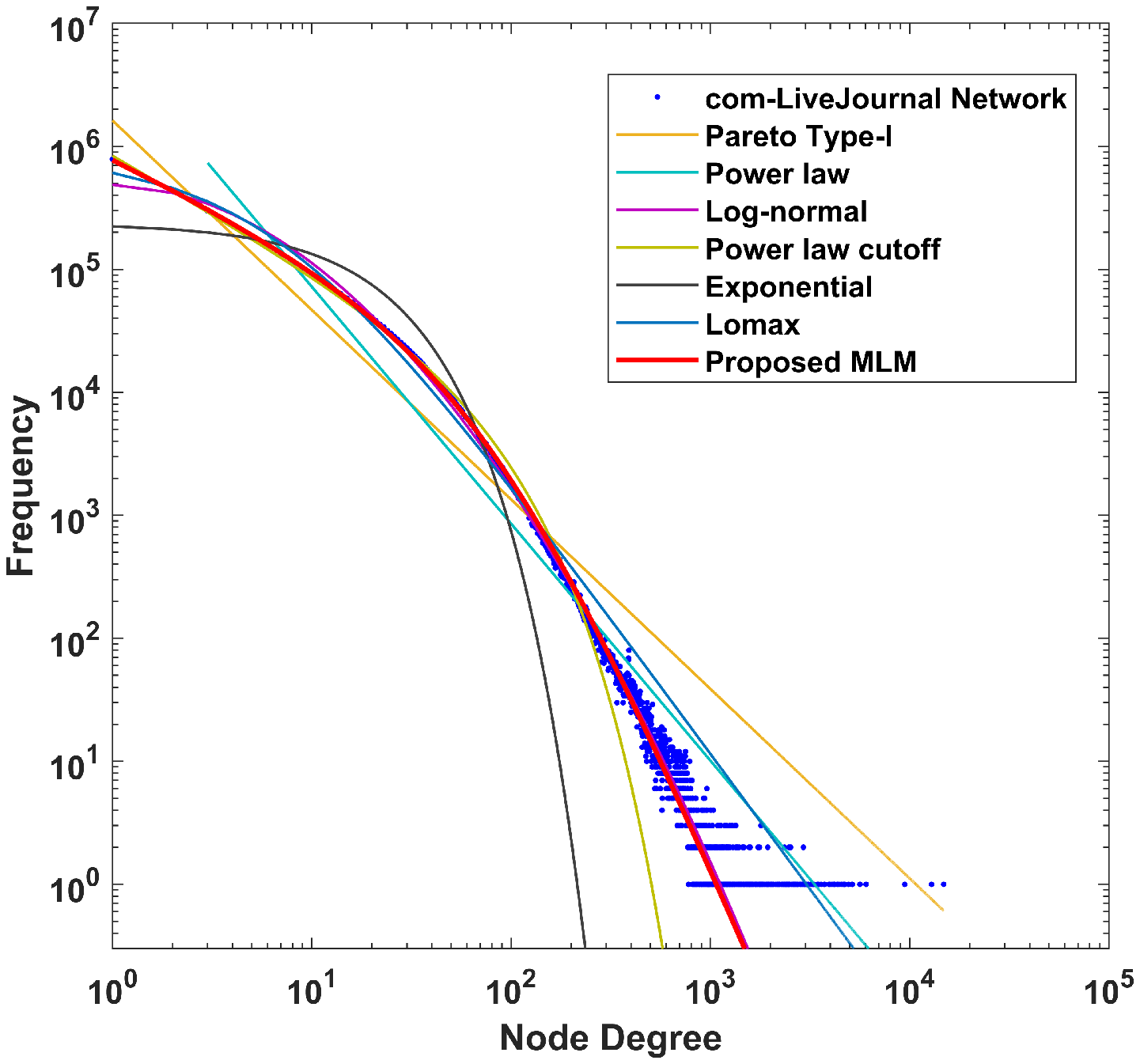}
\includegraphics[width=7cm, height=5cm]{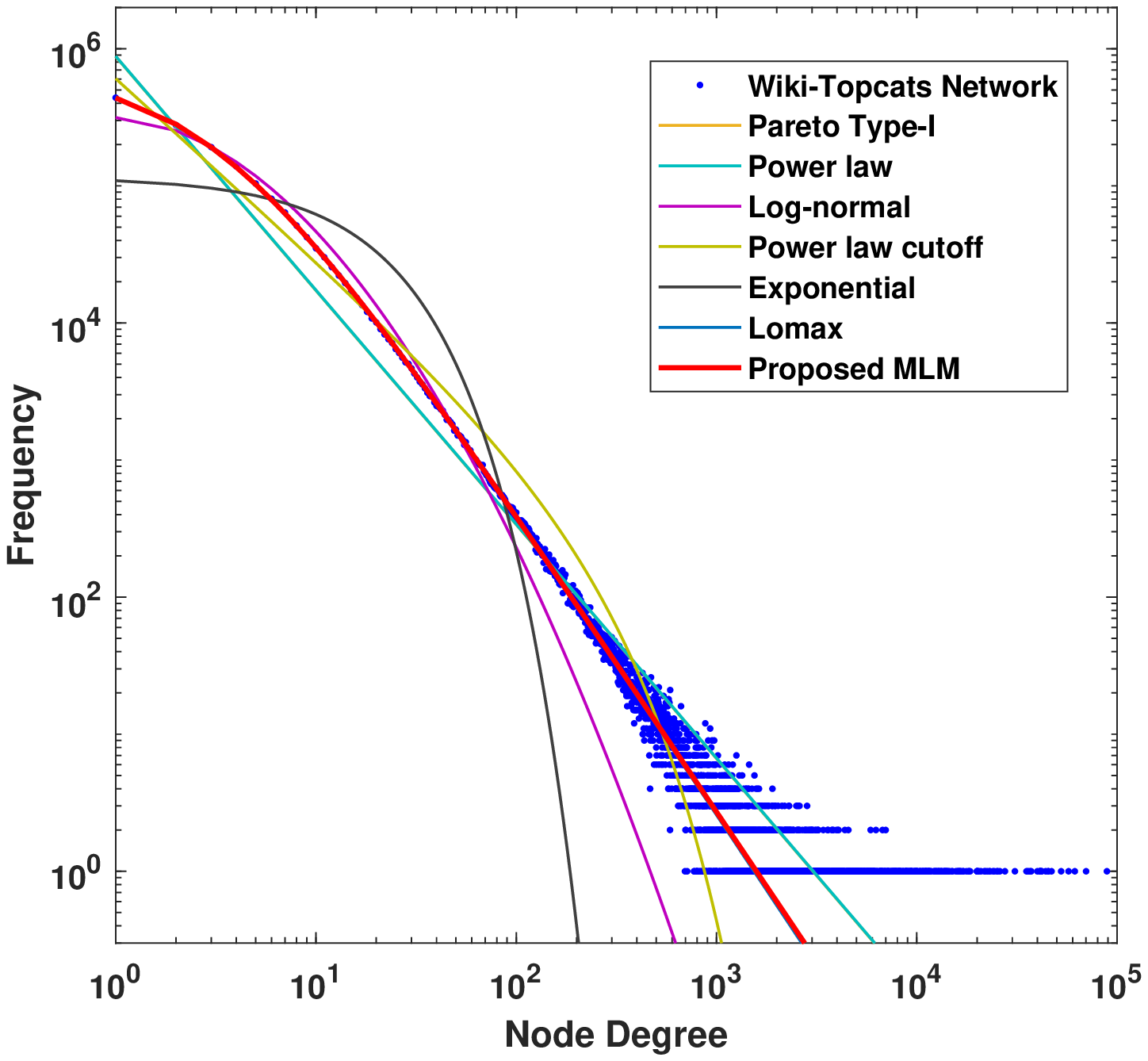}
} \caption{Degree distribution of LiveJournal and Wiki-Topcats
networks in log-log scale} \label{fig_mlm_8}
\end{figure}

The effectiveness of the proposed MLM distribution can also be
verified through the plotting of the fitted results of competitive
distributions. For this purpose, the log-log plots of the of the
original frequency distribution, the estimated frequency by MLM
distribution and the frequency estimated by power-law, pareto,
log-normal, power-law cutoff and exponential distributions are drawn
for all the networks under consideration. Twenty four such examples
have been provided in Figures \ref{fig_mlm_1}-\ref{fig_mlm_8}. These
are the soc-Academia network, ego-Twitter network, Higgs-Twitter
network, ego-Gplus network, cit-HepTh network, cit-Citeseer network,
ca-CondMat network, ca-AstroPh network, Web-Google network,
web-Hudong network, sx-stackoverflow, sx-mathoverflow, Bio-Dmela
network, Bio-Wormnet-V3 network, com-LiveJournal and-
com-Wiki-Topcats network. Few more plotted results are also provided
in the supplementary section. We have omitted the plot of the
poisson distribution due to its poor performances over all the
networks. It is visually clear From Figures
\ref{fig_mlm_1}-\ref{fig_mlm_8} that the proposed MLM distribution
provides better fit compared to the other competitive distributions
in almost all of the networks since the proposed curve always passes
through the middle of the scatter plot of the observed distribution.
In a few cases the power-law cutoff and log-normal provide a better
fit than the proposed distribution. It is visually clear from
observing the social, biological, brain and citation networks that
the entire node degree distribution can be better represented by the
MLM distribution compared to other heavy tailed distributions.
Thus the proposed MLM distribution, a modification of the Lomax
distribution with non linear exponent in the shape parameter, can be
used for effective and efficient modeling of the entire degree
distribution of real-worlld networks without ignoring the lower
degree nodes. The proposed MLM distribution provides more
flexibility in the degree distribution modeling since the
non-negative shape parameter are assumed to be expressed as a
nonlinear function of the data. Empirical results also suggests the
effectiveness of the proposed MLM distribution compared to others as
depicted through Tables
\ref{table:estparameter}-\ref{table:compared2} and Figures
\ref{fig_mlm_1}-\ref{fig_mlm_8}.

\section{Conclusion and Discussion} \label{conclusion}
In this article, we have proposed a modified Lomax (MLM)
distribution derived from a hierarchical family of Lomax
distributions for flexible and efficient modeling of the entire node
degree distribution of real-world complex networks. The proposed MLM
distribution can be thought of as a generalization of the Lomax
distribution with the nonlinear exponent in the shape parameter. We
have theoretically established that the MLM distribution is
heavy-tailed and right-tailed equivalent to the power-law
distribution. Furthermore, we have shown a sufficient condition for
the existence of the MLE for the parameters of MLM distribution
using the notion of CV. The proposed MLM distribution can find MLE
for the parameters at finite points when the value of CV$ > 1$. We
also theoretically justified that the MLM distribution is a function
with regularly varying tails which belongs to the Maximum domain of
attraction of the Frechet distribution. We have further studied the
asymptotic behaviors of the MLM distribution in this context.

The proposed MLM distribution captures the heavy-tailed and
nonlinear behavior of the  entire degree distributions of real-world
networks in the original and the log-log scale more adroitly. It
also enables us to accurately characterize the degree distribution
pattern which may have a significant impact on analyzing real-world
networks in terms of their social or biological aspects, as the case
may be. We have applied the proposed MLM distribution in modeling
the entire degree distribution over 50 different real-world
empirical data sets taken from diverse fields. Empirical results
suggest that as compared to the power-law distribution or any other
well-known distribution, our proposed MLM distribution produces a
lower fitting error in terms of three statistical tests, viz. RMSE,
KL-divergence, and MAE. We also demonstrated the statistical
significance of the estimated MLM distribution with the help of the
bootstrap Chi-square value. This generalization of the Lomax
distribution by adding an additional parameter in the base model
results in flexible modeling to the entire degree distribution of a
real-world network compared to other heavy-tailed distributions
unlike power-law. The proposed fit distribution sometimes helps us
in better characterization of the evolution process of large scale
real-world networks instead of explicitly performing the empirical
study at each time step. Thus, by simulating the parameters of a
proposed fit MLM distribution, one can easily capture the spatial
structure and dynamical pattern of a real-world network as the
network evolves over time. The dynamic pattern analysis of such
structural properties in real-world networks is one of the future
scopes of research.


\section*{Acknowledgement}
The authors gratefully acknowledge the financial assistance received
from Indian Statistical Institute (I. S. I.) and Visvesvaraya PhD
Scheme awarded by the Government of India.

\bibliographystyle{elsarticle-num}
\bibliography{bibliography} 

\end{document}